\documentclass[twoside,a4paper,11pt,intlimits]{elsarticle}

\usepackage{amssymb}
\usepackage{amsmath}
\usepackage{amssymb}
\usepackage{amsthm}

\usepackage{enumitem}
\setenumerate{label={\rm (\alph{*})}}

\usepackage[frak,hyperref,operator]{paper_diening}
\numberwithin{equation}{section} 

\newtheorem{theorem}{Theorem}[section]
\newtheorem{lemma}[theorem]{Lemma}

\newtheorem{corollary}[theorem]{Corollary}

\newtheorem{remark}[theorem]{Remark}

\numberwithin{equation}{section}

\providecommand{\Div}{\divergence}

\providecommand{skp}[2]{{\langle{#1},{#2}\rangle}}

\providecommand{\dx}{\,\mathrm{d}x}
\providecommand{\dy}{\,\mathrm{d}y}

\providecommand{\ds}{\,\mathrm{d}s}
\providecommand{\dt}{\,\mathrm{d}t}
\providecommand{\dtau}{\,\mathrm{d}\tau}
\providecommand{\dxt}{\dx\dt}

\providecommand{\R}{\setR}
\providecommand{\N}{{\mathbb{N}}}
\providecommand{\ep}{\bfvarepsilon}

\providecommand{\Ma}{\ensuremath{\mathcal{M}^\alpha}}
\providecommand{\Mas}{\ensuremath{\mathcal{M}^\alpha_\sigma}}
\providecommand{\Masmk}{\ensuremath{\mathcal{M}^{\alpha_{m,k}}_\sigma}}

\providecommand{\Oal}{\ensuremath{\mathcal{O}^\alpha_\lambda}}

\providecommand{\zal}{\ensuremath{\bfz^\alpha_\lambda}}

\providecommand{\zmk}{\ensuremath{\bfz_{m,k}}}
\providecommand{\umk}{\ensuremath{\bfu_{m,k}}}

\providecommand{ }[1]{\textcolor{blue}{#1}}

\newcounter{formel}

\providecommand{\Bog}{\ensuremath{\text{\rm Bog}}}

\providecommand{\Qz}{\ensuremath{Q_0}}
\providecommand{\Bz}{\ensuremath{B_0}}
\providecommand{\Iz}{\ensuremath{I_0}}

\providecommand{\zQp}{\ensuremath{\bfz_{Q'}}}

\providecommand{\Rdr}{{\setR^3}}

\begin{document}

\begin{frontmatter}

\title{Solenoidal Lipschitz truncation for parabolic PDE's}
  


\author{D.~Breit}
\ead{breit@math.lmu.de}

\author{L.~Diening\corref{cor1}}
\ead{diening@math.lmu.de}

\author{S.~Schwarzacher}
\ead{schwarz@math.lmu.de}

\address{LMU Munich, Institute of Mathematics, Theresienstr. 39,
  80333-Munich, Germany}

\begin{abstract}
  We consider functions $\bfu\in L^\infty(L^2)\cap L^p(W^{1,p})$ with
  $1<p<\infty$ on a time space domain. Solutions to non-linear
  evolutionary PDE's typically belong to these spaces. Many
  applications require a Lipschitz approximation $\bfu_\lambda$ of
  $\bfu$ which coincides with $\bfu$ on a large set. For problems
  arising in fluid mechanics one needs to work with solenoidal
  (divergence-free) functions.  Thus, we construct a Lipschitz
  approximation, which is also solenoidal.  As an application we
  revise the existence proof for non-stationary generalized Newtonian
  fluids of Diening, \Ruzicka{} and Wolf\cite{DieRuzWol10}. Since
  $\divergence \bfu_\lambda=0$, we are able to work in the pressure
  free formulation, which heavily simplifies the proof. We also
  provide a simplified approach to the stationary solenoidal Lipschitz
  truncation of Breit, Diening and Fuchs\cite{BreDieFuc12}.
\end{abstract}



\end{frontmatter}

\section{Introduction}
\label{sec:intro}

The purpose of the Lipschitz truncation technique is to approximate a
Sobolev function~$u \in W^{1,p}$ by $\lambda$-Lipschitz
functions~$u_\lambda$ that coincide with~$u$ up to a set of small
measure. The functions $u_\lambda$ are constructed non-linearly by
modifying~$u$ on the level set of the Hardy-Littlewood maximal
function of the gradient~$\nabla u$.  This idea goes back to Acerbi
and Fusco\cite{AceF88}.  Lipschitz truncations are used in various
areas of analysis: calculus of variations, in the existence theory of
partial differential equations, and in the regularity theory. We refer
to~Ref.~\cite{DieMS08} for a longer list of references.

We are interested in the motion of incompressible fluids. The
balance of momentum reads in the stationary case as
\begin{align}
  \label{0.1}
  \Div \bfS =(\nabla \bfv)\bfv + \nabla \pi -\bff,
\end{align}
where $\bfv$ is the velocity, $(\nabla \bfv)\bfv:=(v_i \partial_i v_j
)_{1\leq j\leq n}$ denotes the convective term, $\bfS$ is the stress
deviator, $\pi$ the pressure, and $\bff$ is the external force. In
order to prescribe the properties of a special fluid one needs a
constitutive law which relates $\bfS$ and the symmetric gradient
$\ep(\bfv):=\frac{1}{2}\big(\nabla\bfv+\nabla \bfv^T\big)$ of the
velocity field.  The most common model for Non-Newtonian fluids
is\cite{AstMar74,BirArmHas87}
\begin{align}
  \label{0.2}
  \bfS =(\kappa+|\ep(\bfv)|)^{p-2}\ep(\bfv),
\end{align}
where $\kappa\geq0$ and $1<p<\infty$. Such fluids are sometimes called
generalized Newtonian fluids. From the mathematical point of view this
problem was firstly investigated by
Ladyzhenskaya\cite{Lad67,Lad68,Lad69} and Lions\cite{Lio69} in the
late sixties.  The existence of a weak solution $\bfv\in
W^{1,p}_{0,\Div}(\Omega)$ to (\ref{0.1})-(\ref{0.2}) is today quite
standard provided $p>\tfrac{3n}{n+2}$.  Here the solution is an
admissible test function to the weak equation and one can directly
apply the monotone operator theory. For smaller values of $p$ the
Lipschitz truncation was firstly used in the fluid context in
Ref.~\cite{FreMS03}, the existence of a weak solution was shown
provided $p>\tfrac{2n}{n+2}$. The idea is to rewrite $(\nabla
\bfv)\bfv$ as $\Div(\bfv\otimes\bfv)$ (using $\Div\bfv=0$) and apply
the Lipschitz truncation to our test function.  The technique was
improved in Ref.~\cite{DieMS08}, where also electro-rheological
fluids were considered. In this paper the following estimate is shown:
\begin{align*}
  \|\nabla \bfu_\lambda\chi_{\set{\bfu\neq\bfu_\lambda}}\|_p\leq
  \delta(\lambda),
\end{align*}
where $\delta(\lambda)\rightarrow0$ for a suitable sequence
$\lambda\rightarrow\infty$. This implies the convergence
$\bfu_\lambda\rightarrow\bfu$ (for $\lambda\rightarrow\infty$) in
$W^{1,p}$ which does not follow from the results of
Acerbi-Fusco\cite{AceF88}.

For the system (\ref{0.1}) it is often convenient to work with the so
called pressure free formulation. This is achieved by the use of
solenoidal (i.e. divergence-free) test functions, since they are
orthogonal to the pressure gradient. The problem of the standard
Lipschitz truncation is, that it does not preserve the solenoidal
property.  The easiest strategy to overcome this deficit is to correct
the functions $\bfu_\lambda$ by means of the \Bogovskii{} operator.
This operator works nice in the uniform convex setting, i.e. on $L^p$
with $1<p<\infty$. However, it cannot be used in the non-uniform
convex setting, e.g. $L^1$, $L^\infty$ or $L^h$ with $h(t)=t
\ln(1+t)$, since the \Bogovskii{} correction is a singular integral
operator. So in the limit cases the \Bogovskii{}-corrected Lipschitz
truncation loses some of its important fine properties. This is
particular the case in the setting of Prandtl-Eyring
fluids\cite{Eyr36}, where the constitutive relation reads as
\begin{equation}
  \label{0.3'}\bfS=\frac{\log(1+|\ep(\bfv)|)}{|\ep(\bfv)|}\ep(\bfv).
\end{equation}
To overcome these problems one needs a solenoidal Lipschitz
truncation.  Therefore in Ref.~\cite{BreDieFuc12} a truncation
method was developed which allows to approximate $\bfu\in
W^{1,p}_{\Div}(\Omega)$ by a solenoidal Lipschitz function
$\bfu_\lambda$ without losing the fine properties of the truncation.

Now, let us turn to the parabolic problem: the balance of momentum
reads as
\begin{align}
  \label{0.4}
 -\partial_t\bfu+ \Div \bfS =(\nabla \bfu)\bfu + \nabla \pi -\bff,
\end{align}
and all involved quantities are defined on the parabolic cube
$Q:=(0,T)\times\Omega$.  Here the situation is much more delicate
since the distributional time derivative $\partial_t\bfv$ interacts
with the pressure which also only exists in the sense of
distributions. In Ref.~\cite{DieRuzWol10} it is shown how to get a
weak solution to (\ref{0.4}) provided $p>\tfrac{2n}{n+2}$.  This is
based on a parabolic Lipschitz truncation and a deep understanding of
the pressure. Further results about parabolic Lipschitz truncation are
due to Kinnunen-Lewis\cite{KinLew02}. In addition to the properties
one needs in the stationary setting we need to know what happens with
the term $\langle\partial_t\bfu,\bfu_\lambda\rangle$ (which appears if
one tests the equation with the truncated function $\bfu_\lambda$).
In Ref.~\cite{DieRuzWol10} it is shown that
\begin{align}
  \big|\langle\partial_t\bfu,\bfu_\lambda\rangle\big|\leq
  \delta(\lambda),\label{0.5}
\end{align}
where $\delta(\lambda)\rightarrow0$ if $\lambda\rightarrow\infty$. The
main ingredients are a parabolic \Poincare{}-inequality and a suitable
scaling. The aim of this paper is to develop a Lipschitz truncation
$\bfu_\lambda$ for a function
$\bfu\in L^p( W^{1,p}_{0,\Div}(\Omega))$ which, in addition to the
properties in Refs.~\cite{DieRuzWol10,KinLew02}, is also solenoidal.

The main motivation for doing so, is to develop an existence theory
for (non-stationary) generalized Newtonian fluids which completely
avoids the appearance of the pressure (note that this cannot be done
by a \Bogovskii{}-correction).  This heavily simplifies the existence
proof for generalized Newtonian fluids from Ref.~\cite{DieRuzWol10}, see
Section~\ref{sec:appl}. We expect that our approach will be useful in
the investigation of electro-rheological fluids\cite{Ruz00} and
Prandtl-Eyring fluids\cite{Eyr36}. In this situations it is not
possible to reconstruct the pressure in the right spaces. So the
standard Lipschitz truncation approach will fail.

In this paper we will construct a solenoidal Lipschitz truncation.
Let $\bfu\in L^p(I, W^{1,p}_{0,\Div}(\Omega))$ be a function with
\begin{align*}
  \int_{Q} \partial_t \bfu \cdot \bfxi \dxt &= \int_{Q} \bfG :
  \nabla \bfxi \dxt \qquad \text{for all } \bfxi \in
  C^\infty_{0,\divergence}(Q),
\end{align*}
where $C^\infty_{0,\divergence}$ is the space of compactly supported,
smooth, solenoidal function.  Then there is a function $\bfu_\lambda$
with roughly the following properties (see Theorem~\ref{thm:ulwhole}
for a precise formulation).
\begin{enumerate}
\item $\nabla\bfu_\lambda\in  {L^\infty}$ with
  $ {\|\nabla\bfu_\lambda\|_\infty} \leq c\lambda$ and
  $\Div\,\bfu_\lambda=0$.
\item $\bfu_\lambda= \bfu$ a.e. outside a suitable set
  $\mathcal{O}_\lambda$ and
  \begin{align*}
     {
    \big|\langle\partial_t\bfu,\bfu_\lambda\rangle\big| +
    \|\chi_{\mathcal{O}_\lambda} \nabla \bfu_\lambda\|_p^p \leq
    c \lambda^p|\mathcal{O}_{\lambda}|} &\leq
    \delta(\lambda),
  \end{align*}
  with $\delta(\lambda)\rightarrow0$ if $\lambda\rightarrow\infty$.
\end{enumerate}
Let us explain the rough ideas of the construction and some
difficulties: We start with the open set $\mathcal{O}_\lambda$, where
the maximal functions of $\nabla\bfu$ or $\bfG$ is bigger than
$\lambda$.  Consider a Whitney decomposition of $\mathcal{O}_\lambda$
into cubes ${Q_i}$ with a special parabolic scaling. Let ${\phi_i}$ be
a subordinate partition of unity. On $Q \setminus \mathcal{O}_\lambda$
the gradient $\nabla\bfu_\lambda$ is already bounded, so we need to
change the function only on $\mathcal{O}_\lambda$. In
Refs.~\cite{DieRuzWol10,KinLew02} this is done via the following
construction with mean values $\bfu_i=\langle\bfu\rangle_{Q_i}$.
\begin{align*}
  \bfu_\lambda &:=
  \begin{cases}
    \bfu & \qquad \text{on $Q \setminus \mathcal{O}_\lambda$},
    \\
    \sum_{i} \phi_i \bfu_i &\qquad \text{on $
      Q \cap \mathcal{O}_\lambda$}.
  \end{cases}
\end{align*}
Of course $\bfu_\lambda$ is not solenoidal in general. So the first idea is to set
\begin{align*}
  \bfu_\lambda &:=
  \begin{cases}
    \bfu & \qquad \text{on $Q \setminus \mathcal{O}_\lambda$},
    \\
    \curl\Big(\sum_{i} \phi_i \Pi_i \big(\curl^{-1}\bfu\big)\Big)
    &\qquad \text{on $ Q \cap \mathcal{O}_\lambda$},
  \end{cases}
\end{align*}
where $\Pi_i$ is a local linear approximation. This approach is very
useful in the stationary context. It simplifies the construction of a
solenoidal Lipschitz truncation from Ref.~\cite{BreDieFuc12}, which
was based on local \Bogovskii{} projections.  We present this new
approach in Section~\ref{sec:sol_stat}.

However in the non-stationary situation we are confronted with the
following problem: The $L^\infty$-estimates for $\nabla\bfu_\lambda$
require a parabolic \Poincare{} inequality.  This needs an information
about the distributional time-derivative which is connected to the
pressure via the equation of motion (see (\ref{0.4})). It is not
enough to control $\partial_t\bfu$ as a functional on the solenoidal
test-functions. So the construction above leads to a solenoidal
Lipschitz truncation where its time derivative still needs information
about the pressure and is therefore not very useful. The main problem
in our construction is to overcome this difficulty which needs a deep
understanding of the equation, especially the properties of the time
derivative. 

The new solenoidal Lipschitz truncation can be found in
Section~\ref{sec:soltrunc}. In Section~\ref{sec:appl} we revisit the
existence proof for non-stationary motion of generalized Newtonian
fluids in order to present how useful this approximation is.

\section{Solenoidal truncation -- evolutionary case}
\label{sec:soltrunc}

In this section we examine solenoidal functions, whose time
derivative is only well defined via the duality with solenoidal
test functions.

Let $\Qz= \Iz \times \Bz \subset \setR \times \Rdr$ be a space time
cylinder.  Let $\bfu \in L^\sigma(\Iz,
W^{1,\sigma}_{\divergence}(\Bz))$ and $\bfG \in L^\sigma(\Qz)$ satisfy
\begin{align}
  \label{eq:uH}
  \int_{\Qz} \partial_t \bfu \cdot \bfxi \dxt &= \int_{\Qz} \bfG :
  \nabla \bfxi \dxt \qquad \text{for all } \bfxi \in
  C^\infty_{0,\divergence}(\Qz),
\end{align}
where we use the subscript $_\divergence$ to denote the subspace of
solenoidal functions.  The goal of this section is to construct a
solenoidal truncation $\bfu_\lambda$ of $\bfu$ which preserves the
properties of the truncation in Refs.~\cite{KinLew02,DieRuzWol10}.
In these papers, equation~\eqref{eq:uH} is valid for all $\bfxi \in
C^\infty_0(Q)$, so one has more control of the time
derivative~$\partial_t \bfu$. This extra control allows to derive a
parabolic \Poincare{} estimate for $\bfu$ in terms of $\nabla \bfu$
and $\bfG$, see Theorem~B.1 of Ref.~\cite{DieRuzWol10}.

In our situation we are confronted with the problem, that the time
derivative~$\partial_t \bfu$ is only defined as a functional on
solenoidal test functions. Therefore, we have not enough control
of~$\partial_t \bfu$ to derive a parabolic \Poincare{} estimate.  This
problem was overcome in Ref.~\cite{DieRuzWol10} by introducing a
pressure in~\eqref{eq:uH}. This pressure splits into a pressure
related to~$\bfG$ and a time derivative of a harmonic pressure related
to $\partial_t\bfu$. Then the sum of $\bfu$ and the harmonic pressure
solves a system of the form~\eqref{eq:uH} for all test
functions~$\bfxi \in C^\infty_0(Q)$. The truncation technique is then
applied to this sum. We want to avoid the introduction of the
pressure, since it is very inflexible and complicates the application
of the truncation method.

Because $\bfu$ and $\bfxi$ are both solenoidal in~\eqref{eq:uH}, they
can both be written as the curl of a vector field. This will allow us
to rewrite~\eqref{eq:uH} as a system that is valid for all functions.
Since the definition of the curl operator depends on the dimension, we
will restrict ourselves in the following for simplicity to the
case~$n=3$. 

Let us be more precise. First we extend our function $\bfu$ in a suitable
way on the whole space and then apply the inverse curl operator.

Let $\gamma \in C_0^\infty(\Bz)$ with$ \chi_{\frac 12 \Bz} \leq \gamma
\leq \chi_{\Bz}$, where $\Bz$ is a ball. Here we use the convention
that $\frac 12 \Bz$ is the scaled ball with the same center (same for
cylinders).  Let $A$ denote the annulus $\Bz \setminus \frac 12 \Bz$.
Then according to Refs.~\cite{Bog80,DieRS10} there exists a \Bogovskii{}
operator $\Bog_A$ from\footnote{$C^\infty_{0,0}$ is the subspace of
  $C^\infty_0$ whose elements have mean value zero.}
$C^\infty_{0,0}(A) \to C^\infty_0(A)$ which is bounded from $L^q_0(A)
\to W^{1,q}_0(A)$ for all $q \in (1,\infty)$ such that $\divergence
\Bog_A = \identity$.  Define
\begin{align*}
  \tilde{\bfu} &:= \gamma \bfu - \Bog_A( \divergence(\gamma \bfu)) =
  \gamma \bfu - \Bog_A( \nabla\gamma \cdot \bfu).
\end{align*}
Then $\divergence \tilde{\bfu}=0$ on $\Iz \times \Bz$ and
$\tilde{\bfu}(t) \in W^{1,\sigma}_0(\Bz)$, so we can extend
$\tilde{\bfu}$ by zero in space to $\tilde{\bfu} \in L^\sigma(I,
W^{1,\sigma}_{\divergence}(\Rdr))$.  Since $\tilde{\bfu}=\bfu$ on $I
\times \frac 12 \Bz$, we have
\begin{align}
  \label{eq:uH2}
  \int_{\Qz} \partial_t \tilde{\bfu} \cdot \bfxi \dxt &= \int_{\Qz}
  \bfG : \nabla \bfxi \dxt \qquad \text{for all } \bfxi \in
  C^\infty_{0,\divergence}(\tfrac 12\Qz).
\end{align}

On the space $W^{1,\sigma}_{\divergence}(\setR^3)$ with $\sigma>1$ we
define the inverse curl operator $\curl^{-1}$ by
\begin{align*}
  \curl^{-1} \bfg := \curl(\Delta^{-1} \bfg) := \curl\bigg(\int_\Rdr
  \frac{-1}{4\pi \abs{x-y}} \bfg(y)\dy\bigg).
\end{align*}
The definition is correct, as in the sense of distributions
\begin{align*}
  \begin{aligned}
    \curl (\curl^{-1} \bfg) &= \curl \curl (\Delta^{-1} \bfg) =
    (-\Delta + \nabla \divergence) \Delta^{-1} \bfg
    \\
    &= \bfg + \nabla \divergence \bigg( \frac{-1}{4 \pi \abs{\cdot}}
    \ast \bfg\bigg)
    \\
    &= \bfg + \nabla \bigg( \frac{-1}{4 \pi \abs{\cdot}} \ast
    \divergence \bfg\bigg)
    \\
    &= \bfg,
  \end{aligned}
\end{align*}
where we used $\divergence \bfg = 0$ in the last step. Moreover,
\begin{align}
  \label{eq:curlmdivfree}
  \divergence (\curl^{-1} \bfg) &= \divergence \curl (\Delta^{-1}
  \bfg) = 0.
\end{align}
Since $\bfg
\mapsto \nabla^2 (\Delta^{-1} \bfg)$ is a singular integral operator,
we have
\begin{align}
  \label{eq:curl1}
  \norm{\nabla \curl^{-1} \bfg}_s \leq \norm{\nabla^2
    (\Delta^{-1}\bfg)}_s \leq c_s\, \norm{\bfg}_s
\end{align}
for $s \in (1,\infty)$. Analogously, we have
\begin{align}
  \label{eq:curl2}
  \norm{\nabla^2 \curl^{-1} \bfg}_s \leq c_s\, \norm{\nabla \bfg}_s
\end{align}
for $s \in (1,\infty)$.

Now, we define pointwise in time
\begin{align*}
  \bfw := \curl^{-1} (\tilde{\bfu}) = \curl^{-1}\big(   \gamma \bfu -
  \Bog_{\Bz \setminus \frac 12 \Bz}( \nabla\gamma \cdot \bfu) \big).
\end{align*}
Overall, we get the following lemma.
\begin{lemma}
  \label{lem:defw}
  We have 
  \begin{alignat*}{2}
    \curl \bfw &= \tilde{\bfu} = \bfu &\qquad&\text{on $\tfrac 12 \Qz$}
    \\
    \divergence \bfw &=0&\qquad&\text{on $\Rdr$}
  \end{alignat*}
  and
  \begin{align*}
    \norm{\bfw(t)}_{L^s(\Rdr)}&\leq c_s\,
    \norm{\tilde{\bfu}(t)}_{ {L^a}(\Bz)}
    \\
    \norm{\nabla \bfw(t)}_{L^s(\Rdr)} &\leq c_s\,
    \norm{\tilde{\bfu}(t)}_{L^s(\Bz)}
    \\
    \norm{\nabla^2 \bfw(t)}_{L^s(\Rdr)} &\leq c_s\, \norm{\nabla
      \tilde{\bfu}(t)}_{L^s(\Bz)}.
  \end{align*}
  for  {$a=\max\{1,\frac{3s}{3+s}\}$}, $t \in I$ and $s \in (1,\infty)$.
\end{lemma}

Let us derive from~\eqref{eq:uH} the equation for $\bfw$. For $\bfpsi
\in C^\infty_0(\tfrac 12 \Qz)$ we have
\begin{align*}
  \int_{\Qz} \partial_t \bfu \cdot \curl \bfpsi \dxt &= \int_{\Qz}
  \bfG : \nabla \curl \bfpsi \dxt.
\end{align*}
We use $\bfu = \curl \bfw$ and partial integration to get
\begin{align*}
  \int_{\Qz} \partial_t \bfw \cdot \curl \curl \bfpsi \dxt &=
  \int_{\Qz} \bfG : \nabla \curl \bfpsi \dxt.
\end{align*}
Now, because
\begin{align*}
  \int_{\Qz} \bfw \cdot \partial_t\nabla \divergence\bfpsi \dx\dt =
  \int_{\Qz} \divergence \bfw\,\partial_t \divergence\bfpsi \dxt=0
\end{align*}
and $\curl \curl \bfpsi = -\Delta \bfpsi+ \nabla \divergence \bfpsi$ we
gain
\begin{align}
  \label{eq:wttmp}
  \int_{\Qz} \bfw\cdot \partial _t \Delta\bfpsi \dxt &= -
  \int_{\Qz} {\bfG}: \nabla \curl\bfpsi \dxt.
\end{align}
for every $\bfpsi \in C^{\infty}_{0}(\tfrac 12 \Qz)$.  We can rewrite
this as
\begin{align}
  \label{eq:wt}
  \int_{\Qz} \bfw\cdot \partial _t \Delta\bfpsi \dxt &= - \int_{\Qz}
  {\bfH}: \nabla^2\bfpsi \dxt,
\end{align}
with $\abs{\bfG} \sim
\abs{\bfH}$ pointwise.  In particular, in the sense of distributions
we have $\partial_t \Delta \bfw = -\curl \divergence \bfG = -
\divergence \divergence \bfH$.

So in passing from $\bfu$ to $\bfw$ we got a system valid for all test
functions~$\bfpsi \in C^\infty_0(\Qz)$. However, we only get control
of $\partial_t \Delta \bfw$, so that the time derivative of the harmonic
part of~$\bfw$ cannot be seen. Hence, a parabolic \Poincare{}
inequality for~$\bfw$ still does not hold; i.e. $\partial_t \bfw$ is
not controlled! In order to remove this invariance we will
replace~$\bfw$ by some function~$\bfz$ such that $\partial_t
\Delta \bfw = \partial_t \Delta \bfz$. This will imply that
$\partial_t \bfz$ can be controlled by~$\bfH$. To define $\bfz$
conveniently we need some auxiliary results.

For a ball $B' \subset \Rdr$ and a function $f \in L^s(B')$ we define $\Delta^{-2}_{B'} \Delta f$ as the weak solution $F \in
W^{2,s}_0(B')$ of
\begin{align}
  \label{eq:defbi2}
  \int_{B'}\Delta F\Delta\phi\dx &= \int_{B'}f\Delta\phi\dx \qquad
  \text{for all }\phi \in C^\infty_0(B').
\end{align}
Then $f - \Delta (\Delta^{-2}_{B'} \Delta f)$ is harmonic on~$B'$.

According to Ref.~\cite{Mul95} and Lemma~2.1 of
Ref.~\cite{Wol07} we have the following variational
estimate.
\begin{lemma}
  \label{lem:mueller}
  Let $s \in (1,\infty)$. Then for all $g \in W^{2,s}_0(B')$ we have
  \begin{align*}
    \norm{\nabla^2 g}_s &\leq c_s \sup_{\substack{\phi \in C^\infty_0(B')
        \\ \norm{\nabla^2 \phi}_{s'} \leq 1}} \int_{B'} \Delta g \Delta
    \phi \dx
  \end{align*}
\end{lemma}
This implies the following two corollaries:
\begin{corollary}
  \label{cor:mueller}
  Let $s \in (1,\infty)$. Then
  \begin{align*}
    \int_{B'} \bigabs{\nabla^2 (\Delta^{-2}_{B'} \Delta f)}^s \dx\leq c_s
    \sup_{\substack{\phi \in C^\infty_0(B') \\ \norm{\nabla^2
          \phi}_{s'} \leq 1}} \int_{B'} f \Delta \phi \dx \leq
    c_s\int_{B'}|f|^s\dx
  \end{align*}
  for $f \in L^s(B')$, where $c_s$ is independent of the ball~$B'$.
\end{corollary}
\begin{proof}
  The claim follows by Lemma~\ref{lem:mueller},
  \begin{align*}
    \int_{B'}\Delta (\Delta^{-2}_{B'} \Delta f)\Delta\phi\dx &=
    \int_{B'} f\Delta\phi\dx
  \end{align*}
  and H{\"o}lder's inequality.
\end{proof}
\begin{corollary}
  \label{cor:muellerhigh}
  Let $s \in (1,\infty)$, then 
  \begin{align*}
    \int_{\frac 23 B'} \bigabs{\nabla^3 (\Delta^{-2}_{B'} \Delta f)}^s
    \dx &\leq c_s\int_{B'}|\nabla f|^s\dx, &&\text{for $f \in
      W^{1,s}(B')$}
    \\
    \int_{\frac 23 B'} \bigabs{\nabla^4 (\Delta^{-2}_{B'} \Delta f)}^s
    \dx& \leq c_s\int_{B'}|\nabla^2 f|^s\dx, &&\text{for $f \in
      W^{2,s}(B')$},
  \end{align*}
  where $c_s$ is independent of the ball~$B'$.
\end{corollary}
\begin{proof}
  The claim follows from Corollary~\ref{cor:mueller} by standard
  interior regularity theory (difference quotients, localization and
  \Poincare{}).
\end{proof}
 {
  For $V \in L^s(B')$ we define $\Delta^{-2}_{B'} \Div\Div V$ as the
  weak solution $F \in W^{2,s}_0(B')$ of
  \begin{align*}
    \int_{B'}\Delta F\Delta\phi\dx &= \int_{B'}V\nabla^2\phi\dx \qquad
    \text{for all }\phi \in C^\infty_0(B').
  \end{align*}
  Similar to Corollary~\ref{cor:mueller} we get the following result.
\begin{corollary}
  \label{cor:divdiv}
 Let $s \in (1,\infty)$, then 
  \begin{align*}
    \int_{B'} \bigabs{\nabla^2 (\Delta^{-2}_{B'} \Div\Div V)}^s \dx \leq
    c_s\int_{B'}|V|^s\dx
  \end{align*}
  for $V \in L^s(B')$, where $c_s$ is independent of
  the ball~$B'$.
\end{corollary}
}
 The next lemma shows the wanted control of the time derivative.
\begin{lemma}
  \label{lem:zi}
  For a cube $Q'= I' \times B' \subset \Qz$ let $\zQp
  := \Delta \Delta_{B'}^{-2} \Delta \bfw$.  Then for $s \in
  (1,\infty)$ we have 
  \begin{align*}
    \dashint_{Q'} \abs{\zQp}^s\dx &\leq c_s\, \dashint_{Q'}
    \abs{\bfw}^s\dx
    \\
    \dashint_{ { Q'}} \biggabs{\frac{\zQp}{r'}}^s\dx +
    \dashint_{\frac 23 Q'} \abs{\nabla \zQp}^s\dx &\leq c_s\,
    \dashint_{ Q'} \abs{\nabla \bfw}^s\dx
    \\
    \dashint_{ {Q'}} \biggabs{ \frac{\zQp}{(r')^2}}^s\dx +
    \dashint_{\frac 23 Q'} \biggabs{ \frac{\nabla \zQp}{r'}}^s\dx
    + \dashint_{\frac 23 Q'} \abs{\nabla^2\zQp}^s\dx &\leq c_s\,
    \dashint_{Q'} \abs{\nabla^2 \bfw}^s \dx,
    \\
    \dashint_{Q'}|\partial_t \zQp|^s \dxt &\leq c_s\,
    \dashint_{Q'}|\bfH|^s \dxt,
  \end{align*}
  where $r':=r_{B'}$.
\end{lemma}
\begin{proof}

  The estimate of $\zQp$ in terms of $\bfw$ follows directly by
  Corollary~\ref{cor:mueller} and integration over time.  
   {
  The estimate of $\zQp$ in terms of $\nabla \bfw$ and $\nabla^2 \bfw$
  follows from this by \Poincare{} using the fact that we can subtract
  a linear polynomial from $\bfw$ without changing the definition of
  $\zQp$.
  The other estimate for $\nabla \zQp$ and $\nabla^2 \zQp$ follow
  analogously from Corollary~\ref{cor:muellerhigh}.
}

  For all $\rho \in C^\infty_0(I')$ and $\bfphi \in C^\infty_0(B')$
  we get by~\eqref{eq:wt} that
  \begin{align*}
    \int_{I'} \int_{B'} \bfw \Delta \bfphi \dx\, \partial_t \rho \dt
    &= -\int_{I'} \int_{B'} \bfH : \nabla^2 \bfphi \dx\, \rho \dt.
  \end{align*}
  Let $d_t^h$ denote the forward difference quotient in time with step
  size~$h$.  We use $\rho(t) := \dashint_{t}^{t-h} \tilde{\rho}(\tau)
  \dtau$ with $\tilde{\rho} \in C^\infty_0(I')$ and $h$ sufficiently
  small. Then $\partial_t \rho = d_t^{-h} \tilde{\rho}$ and
  \begin{align*}
    \int_{I'} \int_{B'} \bfw \Delta \bfphi \dx\, d_t^{-h} \tilde{\rho}
    \dt &= -\int_{I'} \int_{B'} \bfH : \nabla^2 \bfphi \dx\,
    \dashint_{t}^{t-h} \tilde{\rho}(\tau) \dtau
    \dt.
  \end{align*}
  This implies
  \begin{align*}
    \int_{I'} \int_{B'} d_t^h\bfw \Delta \bfphi \dx\, \tilde{\rho}
    \dt &= -\int_{I'} \int_{B'} \dashint_{t}^{t+h} \bfH(\tau) \dtau :
    \nabla^2 \bfphi \dx\, \tilde{\rho} \dt.
  \end{align*}
  Since this is valid for all choices of $\tilde{\rho}$ we have
  \begin{align*}
    \int_{B'} d_t^h\bfw \Delta \bfphi \dx\,  &= -
    \int_{B'} \dashint_{t}^{t+h} \bfH(\tau) \dtau : \nabla^2
    \bfphi \dx.
  \end{align*}
  Since $d_t^h \zQp = d_t^h (\Delta \Delta_{B'}^{-2} \Delta \bfw) =
  \Delta \Delta_{B'}^{-2} \Delta (d_t^h\bfw)$, it follows by
  Corollary~\ref{cor:mueller} that
  \begin{align*}
    \bigg( \int_{B'} \abs{d_t^h \zQp}^s \dx\bigg)^{\frac 1s}
    &\leq c \sup_{\substack{\bfphi \in C^\infty_0(B) \\ \norm{\nabla^2
          \bfphi}_{s'} \leq 1}} \int_{B'} d_t^h \bfw \Delta \bfphi
    \dx
    \\
    &= \sup_{\substack{\bfphi \in C^\infty_0(B) \\ \norm{\nabla^2
          \bfphi}_{s'} \leq 1}} \Bigg(-\int_{B'} \dashint_{t}^{t+h}
    \bfH(\tau) \dtau : \nabla^2 \bfphi \dx \Bigg)
    \\
    &\leq c\, \bigg( \int_{B'} \dashint_{t}^{t+h} \abs{\bfH(\tau)}^s
    \dtau \dx \bigg)^{\frac 1s}.
  \end{align*}
  Integrating over time and passing to the limit $h\to 0$ yields
  \begin{align*}
    \bigg( \int_{Q'} \abs{\partial_t \zQp}^s \dxt \bigg)^{\frac 1s}
    &\leq c\, \bigg( \int_{Q'} \abs{\bfH}^s \dxt \bigg)^{\frac
      1s}. 
  \end{align*}
\end{proof}
We define $\bfz(t) := \bfz_{\frac 12 \Qz}(t) = \Delta \Delta_{\frac 12
  \Bz}^{-2} \Delta \bfw(t)$ for $t \in \frac12 \Iz$, then
\begin{align}
  \label{eq:zt}
  \int_{\Qz} \bfz\cdot \partial _t \Delta\bfpsi \dxt &= \int_{\Qz}
  \bfw\cdot \partial _t \Delta\bfpsi \dxt = - \int_{\Qz} {\bfH}:
  \nabla^2\bfpsi \dxt,
\end{align}
for all $\bfpsi \in C^\infty_0(\tfrac 12 \Qz)$.  Since $\Delta_{\frac
  12 \Bz}^{-2} \bfw(t) \in W^{2,s}_0(\frac 12 \Bz)$, we can extend it
by zero to a function from $W^{2,s}(\Rdr)$. In this sense it is natural
to extend $\bfz(t)$ by zero to a function $L^s(\Rdr)$.

 Note that
Lemma~\ref{lem:zi} enables us to control $\partial_t \bfz$ by $\bfH$
in $L^s(\frac 12 \Qz)$.

\begin{lemma}
  \label{lem:defz}
  We have 
  \begin{align*}
    \norm{\bfz(t)}_{L^s(\frac 13 \Bz)} &\leq c_s\,
    \norm{\tilde{\bfu}(t)}_{L^{\frac{3s}{s+3}}(\Bz)}
    \\
    \norm{\nabla \bfz(t)}_{L^s(\frac 13 \Bz)} &\leq c_s\,
    \norm{\tilde{\bfu}(t)}_{L^s(\Bz)}
    \\
    \norm{\nabla^2 \bfz(t)}_{L^s(\frac 13 \Bz)} &\leq c_s\,
    \norm{\nabla \tilde{\bfu}(t)}_{L^s(\Bz)}.
  \end{align*}
  for $t \in I$ and $s \in (1,\infty)$.
\end{lemma}
\begin{proof}
  This follows from Corollary~\ref{cor:mueller},
  Corollary~\ref{cor:muellerhigh} and Lemma~\ref{lem:defw}.
\end{proof}

Let $\alpha>0$. We say that $Q' = I' \times B' \subset \setR \times
\Rdr$ is an $\alpha$-parabolic cylinder if $r_{I'} = \alpha\,
r_{B'}^2$. For $\kappa>0$ we define the scaled cylinder $\kappa Q' :=
(\kappa I') \times (\kappa B')$.  By $\mathcal{Q}^\alpha$ we denote
the set of all $\alpha$-parabolic cylinders.  We define the
$\alpha$-parabolic maximal operators $\Ma$ and $\Ma_s$ for $s \in
[1,\infty)$ by
\begin{align*}
  (\Ma f)(t,x) &:= \sup_{ Q' \in
    \mathcal{Q}^\alpha\,:\,(t,x) \in Q'} \dashint_{Q'} \abs{f(\tau,y)} \dtau \dy,
  \\
  \Ma_s f(t,x) &:= \big(\Ma(\abs{f}^s(t,x))\big)^{\frac 1s}
\end{align*}
It is standard\cite{Ste93} that for all $q \in (s,\infty]$
\begin{align}
  \label{eq:Mascont}
  \norm{\Ma_s f}_{L^q(\R^{n+1})} &\leq c\, \norm{f}_{L^q(\R^{n+1})}.
\end{align}
For $\lambda,\alpha>0$ and $\sigma>1$ we define
 {
\begin{align}
  \label{eq:defOal}
  \Oal &:= \set{\Mas(\chi_{\frac 13 \Qz} \abs{\nabla^2
      \bfz})>\lambda} \cup \set{\alpha \Mas(\chi_{\frac 13 \Qz}
    \abs{\partial_t\bfz}) > \lambda}.
\end{align}
Later we will choose $\alpha=\lambda^{2-p}$ and $\sigma$ smaller than
the integrability exponent of~$\partial_t\bfz$.}

We want to redefine $\bfz$ on~$\Oal$. The first step is to
cover~$\Oal$ by well selected cubes.
By the lower-semi-continuity property of the maximal functions the set
$\Oal$ is open. We assume in the following that $\Oal$ is
non-empty. (In the case that $\Oal$ is empty, we do not need to
truncate at all.)

According to Lemma~3.1 of Ref.~\cite{DieRuzWol10} there exists an
$\alpha$-parabolic Whitney covering $\set{Q_i}$ of~$\Oal$ in the
following sense:
\begin{enumerate}[label={\rm (PW\arabic{*})}, leftmargin=*]
\item\label{itm:whit1} $\bigcup_{i}\frac {1} {2} Q_i \,=\, \Oal$,
\item\label{itm:whit2} for all $i\in \setN$ we have $8
  Q_i \subset \Oal$ and $16 Q_i \cap (\setR^{n+1}\setminus
  \Oal)\neq \emptyset$, 
\item \label{itm:whit3} if $ Q_i \cap Q_{j} \neq \emptyset
  $ then $ \frac 12 r_j\le r_i< 2\, {r_j}$,
\item \label{itm:whit4}  at every point at most
  $120^{n+2}$ of the sets $4Q_i$ intersect,
\end{enumerate}
where $r_i := r_{B_i}$, the radius of $B_i$ and $Q_i = I_i \times B_i$.

For each $Q_i$ we define $A_i:= \set{ j \,:\, Q_j\cap
  Q_i\neq\emptyset}$. Note that $\# A_i\le 120 ^{n+2}$ and $r_j \sim
r_i$ for all $j \in A_i$. With respect to the covering $\set{Q_i}$
there exists a partition of unity $\set{\phi_i} \subset
C^\infty_0(\R^{n+1})$ such that
\begin{enumerate}[label={\rm (PP\arabic{*})}, leftmargin=*]
\item \label{itm:P1} $\chi_{\frac{1}{2}Q_i}\leq \phi_i\leq
  \chi_{\frac{2}{3}Q_i}$,
\item \label{itm:P2} $\sum_j\phi_j=\sum_{j\in A_i}\phi_j=1$ on $Q_i$,
\item \label{itm:P3} $\abs{\phi_i} + r_i \abs{\nabla \phi_i} + r_i^2
  \abs{\nabla^2 \phi_i} + \alpha\, r_i^2 \abs{\partial_t \phi_i} \leq c$.
\end{enumerate}

Due to property~\ref{itm:whit3} we have that $16 Q_j \cap (\setR^{n+1}
\setminus \Oal) \not=\emptyset$. Thus, the definition of $\Oal$
implies that
\begin{align}
  \label{eq:Qj1}
  \bigg( \dashint_{16 Q_j} \abs{\nabla^2 \bfz}^\sigma \chi_{\frac13\Qz}
  \dxt\bigg)^{\frac 1\sigma} &\leq \lambda,
  \\
  \label{eq:Qj2}
  \alpha\,\bigg( \dashint_{16 Q_j} \abs{ {\partial_t\bfz}}^\sigma \chi_{\frac13\Qz} \dxt\bigg)^{\frac 1
    \sigma} &\leq \lambda.
\end{align}

\begin{lemma}
  \label{lem:Qilocal}
  Assume that there exists $c_0>0$ such that $\lambda^p \abs{\Oal}
  \leq c_0$ with $p > \frac{2n}{n+2}$. Then the following holds:

  If $\lambda \geq \lambda_0=\lambda_0(c_0)$, $\alpha = \lambda^{2-p}$
  and $Q_i \cap \tfrac 14 \Qz \neq \emptyset$, then $Q_i \subset \frac
  13 \Qz$ and $Q_j \subset \tfrac 13 \Qz$ for all $j \in A_i$.
\end{lemma}
\begin{proof}
  Let $Q_i \cap \tfrac 14 \Qz \neq \emptyset$. We claim that $Q_i
  \subset \frac{7}{24} \Qz \subset \frac 13 \Qz$. Let $s_i := \alpha
  r_i^2$. It suffices to show that $r_i, s_i \to 0$ for $\lambda \to
  \infty$. Because $Q_i \subset \Oal$ and by assumption, we find 
  \begin{align}
    \label{eq:Qilocal}
    \lambda^2 r_i^{n+2} = \lambda^p s_i r_i^n \leq\lambda^p \abs{Q_i}
    \leq \lambda^p \abs{\Oal} \leq c_0.
  \end{align}
  This implies $r_i \leq (c\, c_0 \lambda^{-2})^{\frac 1{n+2}}\to 0$ for
  $\lambda \to \infty$. Moreover, $r_i = s_i^{\frac 12} \alpha^{-\frac
    12}$ and~\eqref{eq:Qilocal} imply
  \begin{align*}
    c\, c_0 &\geq \lambda^p s_i r_i^n = \lambda^p s_i^{\frac {n+2}2}
    \alpha^{-\frac 32} = \lambda^{\frac{n+2}{2} p - n} s_i^{\frac
      {n+2}2}.
  \end{align*}
  If $p > \frac{2n}{n+2}$, then $\lambda \to \infty$ implies $s_i \to
  0$ as desired.

  The claim on $j \in A_i$ follows by the fact that $Q_i$ and $Q_j$
  have comparable size and that $\frac{7}{24} \Qz$ is strictly
  contained in $\frac 13 \Qz$.
\end{proof}
Let us show that the assumption $\lambda^p \abs{\Oal} \leq c_0$ from
Lemma~\ref{lem:Qilocal} is satisfied in our situation. For this we
assume from now on that
\begin{align}
  \label{eq:defalpha}
  \alpha := \lambda^{2-p}
\end{align}
and that $\sigma<\min\{p,p'\}.$
\begin{lemma}
  \label{eq:lem:assOal}
  Let $c_0:= \norm{\nabla^2 \bfz}_{L^p(\frac 13 Q_0)}^p +
  \norm{ {\partial_t\bfz}}_{L^{p'}(\frac 13 \Qz)}^{p'}$. Then
  \begin{align*}
    \lambda^p \abs{\Oal} \leq c_0.
  \end{align*}
\end{lemma}
\begin{proof}
  If follows from the weak-type estimate
  of $\Mas$, provided $\sigma<\min\{p,p'\}$ that
  \begin{align*}
    \abs{\Oal} &\leq c\, \lambda^{-p} \norm{\nabla^2 \bfz}_{L^p(\frac
      13 \Qz)}^p + c\, (\lambda \alpha^{-1})^{-p'}
    \norm{ {\partial_t\bfz}}_{L^{p'}(\frac 13 \Qz)}^{p'}
    \\
    &= c\, \lambda^{-p} \big( \norm{\nabla^2 \bfz}_{L^p(\frac 13
      \Qz)}^p + \norm{ {\partial_t\bfz}}_{L^{p'}(\frac 13 \Qz)}^{p'} \big).
  \end{align*}
\end{proof}
In the following we define $\lambda_0$ such that the conclusion of
Lemma~\ref{lem:Qilocal} is valid and assume $\lambda \geq
\lambda_0$. Without loss of generality we can assume further that
\begin{align}
  \label{eq:wlogl0}
  \lambda_0 &\geq \bigg( \dashint_{\frac 13 \Qz} \abs{\nabla^2
    \bfz}^\sigma \dxt \bigg)^{\frac 1\sigma} + r_0^{-2} \bigg(
  \dashint_{\frac 13 \Qz} \abs{\bfz}^\sigma \dxt \bigg)^{\frac
    1\sigma}.
\end{align}

We define
\begin{align*}
  \mathcal{I} := \set{ i \,:\, Q_i \cap \tfrac 14 \Qz \neq \emptyset}.
\end{align*}
Then Lemma~\ref{lem:Qilocal} implies $Q_i \subset \frac 13 \Qz$ (and
$Q_j \subset \frac 13 \Qz$ for $j \in A_i$) for all $i \in
\mathcal{I}$.  
 {
  For each $i \in \mathcal{I}$ we define local approximation $\bfz_i$ for $\bfz$ on $Q_i$ by
  \begin{align}
    \label{eq:def_zj}
    \bfz_i &:=
    \Pi^0_{I_i}\Pi^1_{B_i}(\bfz),
  \end{align}
  where $\Pi^1_{B_j}(\bfz)$ is the first order averaged Taylor
  polynomial\cite{BreSco94,DieR07} with respect to space and $\Pi^0_{I_i}$ is the zero order
averaged Taylor polynomial in time.}
 {Note that this definition implies the \Poincare{}-type inequality.
\begin{lemma}
  \label{lem:SP}
  For all $j\in\setN$ we find for $1\leq s< \infty$ if
  $\nabla^2\bfz,\partial_t \bfz\in L^s(\frac 14\Qz)$, that
  \[
  \dashint_{Q_j}\biggabs{\frac{\bfz-\bfz_j}{r_j^2} }^s \dx\dt +
  \dashint_{Q_j}\biggabs{\frac{\nabla(\bfz-\bfz_j)}{r_j} }^s \dx\dt
  \leq c\, \dashint_{Q_j} \abs{\nabla^2 \bfz}^s\dx\dt +
  c\alpha^s\dashint_{Q_j} \abs{\partial_t \bfz}^s\dx\dt.
  \]
\end{lemma}
\begin{proof}
  The estimate is a consequence of Fubini's Theorem, \Poincare{} estimates and
  the properties of the averaged Taylor polynoms see Lemma~3.1 of
  Ref.~\cite{DieR07}. We find
  \begin{align*}
    \dashint_{Q_j}\biggabs{\frac{\bfz-\bfz_j}{r_j^2} }^s \dx\dt &\leq
    c\dashint_{Q_j}\biggabs{\frac{\bfz-\Pi^1_{B_j}(\bfz)}{r_j^2} }^s
    \dx\dt+c\dashint_{B_j}
    \dashint_{I_j}\biggabs{\frac{\Pi^1_{B_j}(\bfz)-
        \Pi^0_{I_j}\Pi^1_{B_j}(\bfz)}{r_j^2} }^s \dx\dt
    \\
    &\leq c\dashint_{Q_j} \abs{\nabla^2
      \bfz}^s\dx\dt+c\alpha\dashint_{I_j}\dashint_{B_j}
    \abs{\partial_t\Pi^1_{B_j}(\bfz) }^s\dx\dt.
  \end{align*}
  Now the continuity of $\Pi^1_{B_j}$ on $L^s$ gives the estimate.
  Similar we find as all norms for polynomials are equivalent
  \begin{align*}
    \dashint_{Q_j}\biggabs{\frac{\nabla(\bfz-\bfz_j)}{r_j} }^s \dx\dt
    &\leq
    c\dashint_{Q_j}\biggabs{\frac{\nabla(\bfz-\Pi^1_{B_j}(\bfz))}{r_j}
    }^s \dx\dt+ c\dashint_{Q_j}\biggabs{\frac{\nabla(\Pi^1_{B_j}(\bfz)-
        \Pi^0_{I_j} \Pi^1_{B_j}(\bfz))}{r_j} }^s \dx\dt
  \\
&\leq
    c\dashint_{Q_j}\biggabs{\frac{\nabla(\bfz-\Pi^1_{B_j}(\bfz))}{r_j}
    }^s \dx\dt+ c\dashint_{Q_j}\biggabs{\frac{\Pi^1_{B_j}(\bfz)-
        \Pi^0_{I_j} \Pi^1_{B_j}(\bfz)}{r_j^2} }^s \dx\dt
\\
  &\leq c\dashint_{Q_j} \abs{\nabla^2
    \bfz}^s\dx\dt + c\alpha\dashint_{Q_j} \abs{\partial_t \bfz}^s\dx\dt.
\end{align*}
\end{proof}
We can now define our truncation $\zal$ for $\lambda \geq \lambda_0$
on $\frac14 \Qz$ by
\begin{align}
  \label{eq:defwl}
  \zal &:= \bfz - \sum_{i \in \mathcal{I}} \phi_i (\bfz - \bfz_i).
\end{align}
It suffices to sum over~$i$ with $Q_i
\cap \frac 14 \Qz \neq \emptyset$.}

Since the $\phi_i$ are locally finite, this sum is pointwise
well-defined. We will see later that the sum converges also in other
topologies.  Using $\sum_{i \in \mathcal{I}} \phi_i = 1$ on $\frac 14
\Qz$, we can write $\zal$ also in the following form.
%
\begin{align}
  \label{eq:defwlalt}
  \zal &=
  \begin{cases}
    \bfz & \qquad \text{on $\tfrac 14 \Qz \setminus \Oal$},
    \\
    \sum_{i \in \mathcal{I}} \phi_i  {\bfz_i} &\qquad \text{on $\tfrac 14
      \Qz \cap \Oal$}.
  \end{cases}
\end{align}
In the
following we provide some qualities of the truncation (e.g.
$ {\nabla^2} \zal\in L^\infty(\frac14 \Qz)$).
 {
\begin{lemma}
  \label{lem:SP2}
  For all $j \in \setN$ and all $k \in \setN$ with $Q_j \cap Q_k \neq \emptyset$ we have
  \begin{enumerate}[label={\rm (\alph{*})}]
  \item 
    \label{itm:SP2}
      $\dashint_{Q_j} \abs{\nabla^2 \bfz} \,\dx\dt +\alpha\dashint_{Q_j} \abs{\partial_t\bfz} \,\dx\dt \leq c\,\lambda$.
  \item 
    \label{itm:SP3} 
    $\norm{\bfz_j - \bfz_k}_{L^\infty(Q_j)} \leq c
    \dashint_{Q_j}\abs{\bfz-\bfz_j}\dx\dt + c
    \dashint_{Q_k}\abs{\bfz-\bfz_k}\dx\dt$.
  \item \label{itm:SP4} $\norm{\bfz_j - \bfz_k}_{L^\infty(Q_j)}
    \leq c\, r_j^2\, \lambda$.
  \end{enumerate}
\end{lemma}
}
 {
\begin{proof} 
  The first part~\ref{itm:SP2}
  follows from $Q_j \subset 16 Q_j$ and $16 Q_j \cap
  \mathcal{O}_\lambda^\complement \not=\emptyset$, so 
\[\dashint_{16
    Q_j} (\abs{\nabla^2 \bfz} +\alpha\abs{\partial_t \bfz})\chi_{\frac 13\Qz} \dx\dt\leq \bigg(\dashint_{16
    Q_j}(\abs{\nabla^2 \bfz} +\alpha\abs{\partial_t \bfz})^\sigma\chi_{\frac 13\Qz} \dx\dt\bigg)^\frac1\sigma \leq c\lambda.
\]
  Part~\ref{itm:SP3} follows from the
  geometric property of the~$Q_j$. If $Q_j \cap Q_k
  \neq\emptyset$, then $\abs{Q_j \cap Q_k} \geq c\, \max
  \{\abs{Q_j}, \abs{Q_k}\}$. This and the norm equivalence for linear
  polynomials imply
  \begin{align*}
    \norm{\bfz_j - \bfz_k}_{L^\infty(Q_j)} &\leq c\, \dashint_{Q_j
      \cap Q_k} \abs{\bfz_j - \bfz_k} \dx\dt
    \\
    &\leq c \dashint_{Q_j}\left|\bfz_j-\bfz\right|\,\dx + c
    \dashint_{Q_k}\left|\bfz-\bfz_k\right|\,\dx.
  \end{align*}
  Finally,~\ref{itm:SP4} is a consequence of Lemma~\ref{lem:SP},
  \ref{itm:SP2} and \ref{itm:SP3}.
\end{proof} 
}
  We begin
by proving the stability of the truncation.   {
\begin{lemma}
  \label{lem:zlLpW2p}
  Let $1<s<\infty$ and $\bfz\in L^s(\setR; W^{2,s}(\Rdr))$.  Then it
  holds
  \begin{align*}
    \|\zal\|_{L^s(\frac 14 \Qz)} &\leq c\, \|\bfz\|_{L^s(\frac 13
      \Qz)}, 
    \\
    \|\nabla \zal\|_{L^s(\frac 14 \Qz)} &\leq c\, \|\nabla
    \bfz\|_{L^s(\frac 13 \Qz)}+c\alpha r_0\|\partial_t
    \bfz\|_{L^s(\frac 13 \Qz)},
    \\
    \|\nabla^2\zal\|_{L^s(\frac 14 \Qz)} +\alpha \|\partial_t
    \zal\|_{L^s(\frac 13 \Qz)}&\leq c\,
    \|\nabla^2\bfz\|_{L^s(\frac 13 \Qz)}+c\alpha \|\partial_t
    \bfz\|_{L^s(\frac 13 \Qz)}.
  \end{align*}
  { Moreover, the sum in~\eqref{eq:defwl} converges in $L^s(\frac 14 \Iz,
  W^{2,s}(\frac 14 \Bz))$.}
\end{lemma}
}
 {
\begin{proof}
 {
  We show first that the sum in~\eqref{eq:defwl} converges absolutely
  in $L^s(\frac 14 \Qz)$:
  \begin{align*}
    \int_{\frac 14 \Qz} \abs{\bfz-\zal}^s \dx &\leq \sum_{i \in
      \mathcal{I}} \int_{Q_i} \abs{\bfz-\bfz_i}^s\dxt \leq c \sum_{i \in
      \mathcal{I}} \int_{Q_i} \abs{\bfz}^s\dxt \leq c
    \int_{\frac 13 \Qz} \abs{\bfz}^s \dt,
  \end{align*}
  where we used continuity of the mapping $\bfz\mapsto \bfz_i$ in $L^s(Q_i)$, \ref{itm:P1} and the finite intersection property of $Q_i$ \ref{itm:whit4}.}
 {
We start by showing the estimate for the second derivatives
 \begin{align*}
   \int_{\Oal}&\abs{\nabla^2(\bfz-\zal)}^s\dx\dt = \biggabs{\sum\limits_{i \in
 \mathcal{I}} \int_{Q_i}\nabla^2(\phi_i(\bfz-\bfz_i))\dx\dt}\\
&\leq c\sum\limits_{i \in
 \mathcal{I}} \int_{Q_i}\abs{\nabla^2\bfz}^s+\biggabs{\frac{\nabla(\bfz-\bfz_i)}{r_i}}^s+\biggabs{\frac{\bfz-\bfz_i}{r_i^2}}^s\dx\dt.
 \end{align*}
For the time derivative we find as $\bfz_i$ is constant in time, that
\begin{align}
\label{eq:time}
\begin{aligned}
   \int_{\Oal}&\abs{\partial_t(\bfz-\zal)}^s\dx\dt = \biggabs{\sum\limits_{i \in
 \mathcal{I}} \int_{Q_i}\partial_t(\phi_i(\bfz-\bfz_i))\dx\dt}\\
&\leq c\sum\limits_{i \in
 \mathcal{I}} \int_{Q_i}\abs{\partial_t\bfz}^s+\biggabs{\frac{\bfz-\bfz_i}{\alpha r_i^2}}^s\dx\dt.
\end{aligned}
 \end{align}
 {
We apply Lemma~\ref{lem:SP} and use the finite intersection of the $Q_i$ to gain
\begin{align*}
  \int_{\frac 14 \Qz}\abs{\nabla^2(\bfz-\zal)}^s+\alpha^s\abs{\partial_t(\bfz-\zal)}^s\dx\dt &\leq c\sum\limits_{i \in
 \mathcal{I}}\int_{Q_i}\abs{\nabla^2\bfz}^s + \alpha^s\abs{\partial_t \bfz}^s\dx\dt\\
& \leq c\int_{\frac 13 Q_0}\abs{\nabla^2 \bfz}^s+\alpha^s\abs{\partial_t \bfz}^s\dx\dt.
\end{align*}
The estimate of the gradient is analogous, as
\begin{align*}
 \int_{\Oal}\abs{\nabla(\bfz-\zal)}^s\dxt &\leq  \sum\limits_{i \in
 \mathcal{I}}\abs{\nabla(\bfz-\bfz_i)}^s+\biggabs{\frac{\bfz-\bfz_i}{r_i}}^s\dx\dt.
\end{align*}}
}
\end{proof}
}

The truncation $\zal$ has better regularity properties than~$\bfz$; indeed, $\nabla\bfz$ is  {Lipschitz}.
 {
\begin{lemma}
\label{lem:zlLip}
For $\lambda>\lambda_0$ we have
\[
 \norm{\nabla^2 \zal}_{L^\infty(\frac 14\Qz)}+r_0^{-1}{\norm{\nabla \zal}_{L^\infty(\frac 14\Qz)}}+r_0^{-2}\norm{\zal}_{L^\infty(\frac 14\Qz)}+\alpha \norm{\partial_t \zal}_{L^\infty(\frac 14\Qz)}\leq c\lambda.
\]
\end{lemma}
}
 {
\begin{proof}
  Let $(t,x)\in Q_i$, then
\[
 \abs{\nabla^2\zal(t,x)}=\biggabs{\sum\limits_{j\in A_i}\nabla^2(\phi_j\bfz_j)(t,x)}\leq  \sum\limits_{j\in A_i}\abs{\nabla^2(\phi_j(\bfz_j-\bfz_i))(t,x)}
\]
because $\set{\phi_j}$ is a partition of unity. Now we find as all
norms on polynomials are equivalent, $\# A_j\leq c$ and
Lemma~\ref{lem:SP2} that
\[
 \abs{\nabla^2\zal(t,x)}\leq c\sum\limits_{j\in A_i}\frac{\norm{\bfz_i-\bfz_j}_{L^\infty(Q_i)}}{r_i^2}\leq c\lambda.
\]
Concerning the time derivative for $(t,x)\in Q_i$ as $\bfz_i$ is constant in time we find
\begin{align*}
 \abs{\partial_t\zal(t,x)}&=\biggabs{\partial_t\sum\limits_{j\in A_i}(\phi_j\bfz_j)(t,x)}\leq  \sum\limits_{j\in A_i}\abs{\partial_t(\phi_j)(\bfz_j-\bfz_i)(t,x)}\\
&\leq  \sum\limits_{j\in A_i}\frac{\norm{\bfz_i-\bfz_j}_{L^\infty(Q_i)}}{\alpha r_i^2}\leq \frac{c\lambda}{\alpha}.
\end{align*}
 {
The zero order term is estimated by \Poincare inequality; first in time and then in space
\begin{align*}
r_0^{-2}\norm{\zal}_{L^\infty(\frac 14\Iz; L^\infty(\frac 14\Bz))}
\leq c\alpha \norm{\partial_t\zal}_{L^\infty(\frac 14 \Qz)} + cr_0^{-2}\norm{\zal}_{L^1(\frac 14\Iz; L^\infty(\frac 14\Bz))}
\\
\leq c\lambda + 
c\norm{\nabla^2\zal}_{L^\infty(\frac 14 \Qz)}
+cr_0^{-2}\norm{\zal}_{L^1(\frac 14\Qz)}.
\end{align*}
This implies by the norm equivalence of polynomials, Jensen's inequality Lemma~\ref{lem:zlLpW2p} and \eqref{eq:wlogl0}
\[
  r_0^{-1}{\norm{\nabla \zal}_{L^\infty(\frac 14\Qz)}}+r_0^{-2}\norm{\zal}_{L^\infty(\frac 14\Qz)}\leq c\lambda+r_0^{-2}\norm{\bfz}_{L^\sigma(\frac 13\Qz)}\leq c \lambda.
\]
}
\end{proof}
}

The next lemma will control the time error we get when we apply the truncation as a test function.
\begin{lemma}
  \label{lem:zt1}
  For all $\zeta\in C^\infty_0(\frac 14 \Qz)$ with $\norm{\nabla^2
    \zeta}_\infty \leq c$ and $\lambda \geq \lambda_0$
  \begin{align*}
    \biggabs{ \int_{\frac 14 \Qz} \partial _t \big(\bfz -
      \zal\big)\,\Delta(\zeta \zal) \dxt} \leq c\, \alpha^{-1}
    \lambda^2\, \abs{\Oal},
  \end{align*}
  where the constant $c$ is independent of $\alpha$ and $\lambda$.
\end{lemma}

\begin{proof}
 {
We use H{\"o}lder's inequality and Lemma~\ref{lem:zlLip} to gain
  \begin{align*}
    (I) &:= \biggabs{ \int_{\frac 14 \Qz} \partial _t \big(\bfz -
      \zal\big)\,\Delta(\zeta \zal) \dxt}
    \\
    &\leq \sum_{i \in \mathcal{I}} \bigg(\int_{Q_i} \bigabs{\partial_t
      (\phi_i (\bfz-\bfz_i)}^\sigma\dxt\bigg)^\frac1\sigma\bigg(\int_{Q_i}
    \abs{\Delta(\zeta\zal)}^{\sigma'}\dxt\bigg)^\frac1{\sigma'}
    \\
    &\leq c\, \lambda \sum_{i \in \mathcal{I}} \abs{Q_i} \bigg(
    \dashint_{Q_i} \bigabs{\partial_t (\phi_i (\bfz-\bfz_i)}^\sigma \dxt
    \bigg)^{\frac 1\sigma}.
  \end{align*}
}
  With  {\eqref{eq:time}}, \eqref{eq:Qj1} and~\eqref{eq:Qj2}
 we get 
  \begin{align*}
    (I) &\leq c\, \lambda \sum_{i \in \mathcal{I}} \abs{Q_i} \Bigg(
     \alpha^{-1} \bigg( \dashint_{Q_j}
    |\nabla^2\bfz|^\sigma\dxt\bigg)^{\frac 1 \sigma} +
    \bigg(\dashint_{Q_j}| {\partial_t\bfz}|^\sigma \dxt \bigg)^{\frac 1\alpha}
    \Bigg)
    \\
    &\leq c\, \alpha^{-1} \lambda^2 \sum_{i \in \mathcal{I}} \abs{Q_i}
    \leq c\, \alpha^{-1} \lambda^2 \abs{\Oal},
  \end{align*}
  using in the last step the local finiteness of the $\set{Q_i}$.
\end{proof}

\begin{theorem}
  \label{thm:wlwhole}
  Let $1<p< \infty$ with $p,p'>\sigma$.  Let $\bfw_m$ and $\bfH_m$
  satisfy $\partial_t \Delta \bfw_m = - \divergence \divergence
  \bfH_m$ in the sense of distributions~$\mathcal{D}'(\tfrac 12 \Qz)$,
  see~\eqref{eq:wt}.  Further assume that $\bfw_m$ is a weak null
  sequence in $L^p(\frac 12 \Iz; W^{2,p}(\frac 12 \Bz))$ and a strong
  null sequence in $L^\sigma(\frac 12 \Qz)$. Further, assume that
  $\bfH_m = \bfH_{1,m} + \bfH_{2,m}$ such that $\bfH_{1,m}$ is a weak
  null sequence in $L^{p'}(\Qz)$ and $\bfH_{2,m}$ converges strongly
  to zero in $L^\sigma(\Qz)$.  Define $\bfz_m:=\Delta \Delta_{\frac 12
    \Bz}^{-2}\Delta\bfw_m$ pointwise in time on $\frac 12 \Iz$.  Then
  there is a double sequence $(\lambda_{m,k})\subset \R^+$ and $k_0
  \in \setN$ such that
  \begin{enumerate}[label=(\alph{*}),start=1]
  \item \label{itm:wlw:lmk} $2^{2^k} \leq \lambda_{m,k}\leq
    2^{2^{k+1}}$
  \end{enumerate}
  such that the double sequence $\bfz_{m,k} :=
  \bfz^{\alpha_{m,k}}_{\lambda_{m,k}}$ with $\alpha_{m,k} :=
  \lambda_{m,k}^{2-p}$ satisfies the following properties for all $k
  \geq k_0$
  \begin{enumerate}[label=(\alph{*}),start=2]
  \item \label{itm:wlw:setne} $\set{\zmk \neq \bfz} \subset
    \mathcal{O}_{m,k} := \mathcal{O}^{\alpha_{m,k}}_{\lambda_{m,k}}$,
 { 
 \item \label{itm:wlw:n2BMO} $\|\nabla^2\zmk\|_{L^\infty(\frac 14\Qz)}\leq c\,\lambda_{m,k}$,
  \item \label{itm:wlw:nweak} $\zmk \rightarrow 0$ and $\nabla\zmk
    \rightarrow 0$ in $L^\infty(\frac 14 \Qz)$ for $m \to \infty$ and
    $k$ fixed.
  \item \label{itm:wlw:n2weak} $\nabla^2 \zmk \rightharpoonup^* 0$ in
    $L^\infty(\frac 14 \Qz)$ for $m \to \infty$ and $k$ fixed.
    }
  \item \label{itm:wlw:time}We have for all $\zeta\in C^\infty_0(\frac
    14 \Qz)$\\$\displaystyle{ \bigg|\int \big(
      \partial _t \big(\bfz_m-
      \zmk\big)\big)\,\Delta(\zeta\zmk) \dxt\bigg| \leq
      c\,\lambda_{m,k}^p
      \abs{\mathcal{O}_{m,k}}}$,
  \item \label{itm:wlw:levelset} $\displaystyle{
      \limsup_{m\rightarrow\infty}\lambda_{m,k}^p
      \abs{\mathcal{O}_{m,k}}\leq c\,2^{-k}\,
      \sup_m (\norm{\nabla^2 \bfz_m}_p + c\,
      \norm{\bfH_{1,m}}_{p'}^{\frac 1{p-1}}) }$.
  \end{enumerate}
\end{theorem}
\begin{proof}
  Let us assume that $\lambda_{m,k}$ satisfies~\ref{itm:wlw:lmk}. We
  will choose the precise values of~$\lambda_{m,k}$ later. Due to
  Lemma~\ref{lem:zi} we have $\bfz_m \weakto 0$ in $L^p(\frac 14 \Iz;
  W^{2,p}(\frac 14 \Bz))$; this is due to the fact that the operator
  $\bfw\mapsto\Delta \Delta_{\frac 12 \Bz}^{-2}\Delta \bfw =\bfz$ is
  linear and continuous in $L^p(\frac 14 \Iz; W^{2,p}(\frac 14 \Bz))$.
  Then the properties \ref{itm:wlw:setne}
and
  \ref{itm:wlw:n2BMO} follow from Lemma~\ref{lem:zlLip}.  Moreover,
  Corollary~\ref{cor:mueller} ensures that the strong convergence in
  $L^\sigma(\frac 12 \Qz)$ transfers from $\bfw_m$ to $\bfz_m$. By
  Lemma~\ref{lem:zlLpW2p} we get the same for~$\bfz_{m,k}$
    { and that }the sequence $\nabla^2
  \bfz_{m,k}$ is for fixed~$k$ and~$s$ bounded in $L^s(\frac 14 \Qz)$.
  The combination of these convergence properties implies by
  interpolation~\ref{itm:wlw:nweak}.  Moreover, the boundedness of
  $\nabla^2 \zmk$ in $L^s(\frac 14 \Qz)$ implies the weak convergence
  of a subsequence. Since~\ref{itm:wlw:nweak} ensures that the limit
  is zero, we get by the usual arguments the weak convergence of the
  whole sequence. This proves~\ref{itm:wlw:n2weak}.
  Moreover,~\ref{itm:wlw:time} follows by Lemma~\ref{lem:zt1} and the
  choice of~$\alpha_{m,k}$.


  %
  It remains to choose~$2^{2^k}\leq\lambda_{m,k}\leq 2^{2^{k+1}}$ such
  that~\ref{itm:wlw:levelset} holds.  { We use the decomposition
  \begin{align*}
  \partial_t\bfz_m&=\Delta \Delta^{-2}_{\frac12\Bz}\Div\Div\bfH_m=\Delta \Delta^{-2}_{\frac12\Bz}\Div\Div\bfH_{1,m}+\Delta \Delta^{-2}_{\frac12\Bz}\Div\Div\bfH_{2,m}\\
  &=:\bfh_{1,m}+\bfh_{2,m}.
  \end{align*}
  }
We divide
  \begin{align*}
    \mathcal{O}_{m,k} &=
    \set{\Masmk(\chi_{\frac13 \Qz}\abs{\nabla^2 \zmk})>\lambda_{m,k}}
    \cup \set{\alpha_{m,k} \Masmk(\chi_{\frac13 \Qz}\abs{ {\partial_t\bfz_m}}) >
      \lambda_{m,k}}
    \\
    &\subset \set{\Masmk(\chi_{\frac13 \Qz}\abs{\nabla^2
        \zmk})>\lambda_{m,k}} \cup \set{\alpha_{m,k}
      \Masmk(\chi_{\frac13 \Qz}\abs{ {\bfh_{1,m}}}) > \tfrac 12
      \lambda_{m,k}}
    \\
    &\quad \cup \set{\alpha_{m,k} \Masmk(\chi_{\frac13
        \Qz}\abs{ {\bfh_{2,m}}}) > \tfrac 12 \lambda_{m,k}}
    \\
    &=: I \cup II \cup III.
  \end{align*}
  Define 
  \begin{align*}
    g_m := 2\Masmk(\chi_{\frac13 \Qz}\abs{\nabla^2 \bfz_m})+ \Big(2\,
    \Masmk(\chi_{\frac13 \Qz}\abs{ {\bfh_{1,m}}} ) \Big)^{\frac{1}{p-1}}.
  \end{align*}
  Then by the boundedness of $\mathcal{M}_\sigma$ on $L^p$ and
  $L^{p'}$ (using $p,p'>\sigma$)  {as well as Corollary
    \ref{cor:divdiv}} we have
  %
  \begin{align*}
    \norm{g_m}_p &\leq \bignorm{2\Masmk(\chi_{\frac13 \Qz}\abs{\nabla^2 \bfz_m})}_p +
    \bignorm{(2\, \Masmk(\chi_{\frac13 \Qz}\abs{ {\bfh_{1,m}}} ))^{\frac
        1{p-1}}}_p
    \\
    &= \bignorm{2\Masmk(\chi_{\frac13 \Qz}\abs{\nabla^2 \bfz_m})}_p + \bignorm{2\,
      \Masmk(\chi_{\frac13 \Qz}\abs{ {\bfh_{1,m}}} )}_{p'}^{\frac 1{p-1}}
    \\
    &\leq c\,  {\bignorm{\nabla^2 \bfz_m}_{L^p(\frac13\Qz)} + c\,
    \bignorm{\bfh_{1,m}}_{L^{p'}(\frac12\Qz)}^{\frac 1{p-1}}}\\
    & {\leq  c\, \bignorm{\nabla^2 \bfz_m}_{L^p(\frac13\Qz)} 
+ c\,
    \bignorm{\bfH_{1,m}}_{L^{p'}(\frac12\Qz)}^{\frac 1{p-1}}.}
  \end{align*}
  Let $K := \sup_m (\norm{\nabla^2 \bfz_m}_p + c\,
  \norm{ {\bfh_{1,m}}}_{p'}^{\frac 1{p-1}})$.  In particular,
  $\norm{g_m}_p \leq K$ uniformly in~$k$.  Note that
  \begin{align*}
    I \cup II &= \set{\Masmk(\chi_{\frac13 \Qz}\abs{\nabla^2
        \zmk})>\lambda_{m,k}} \cup \set{(\Masmk(\chi_{\frac13
        \Qz}\abs{ {\bfh_{1,m}}}))^{\frac 1{p-1}} > \lambda_{m,k}}
    \\
    &\subset \set{2\Masmk(\chi_{\frac13 \Qz}\abs{\nabla^2
        \zmk})+(2\,\Masmk(\chi_{\frac13 \Qz}\abs{ {\bfh_{1,m}}}))^{\frac
        1{p-1}} > \lambda_{m,k}}
    \\
    &= \set{g_m > \lambda_{m,k}}.
  \end{align*}
  We estimate
  \begin{align*}
    \int_{\R^{n+1}} \abs{g_m}^p \,\dx &= \int_{\R^{n+1}} \int_0^\infty
    \frac{1}{p} t^{p-1} \chi_{\{\abs{g_m}>t\}} \,\dt \,\dx \geq
    \int_{\R^{n+1}} \sum_{k \in \N} \frac{1}{p} 2^k
    \chi_{\{\abs{g_m}>2^{k+1}\}} \,\dx
    \\
    &\geq \sum_{j \in \N} \sum_{k=2^j}^{2^{j+1}-1} \frac{1}{p}
    2^{kp} \abs{\{\abs{g_m}>2^{k+1}\}}.
  \end{align*}
  For fixed $m,j$ the sum over~$k$ involves $2^j$ summands and not all
  of them can be large.  Consequently there exists $\lambda_{m,k} \in
  \set{2^{2^k+1}, \dots, 2^{2^{k+1}}} $, such that 
  \begin{align*}
    \lambda_{m,k}^p \bigabs{ \{ \abs{g_m} >\lambda_{m,k} \}}
    \,\leq\,c \, 2^{-k}\, K^p
  \end{align*}
  uniformly in $m$ and $k$. This proves
  \begin{align}
    \label{eq:4.17b}
    \lambda_{m,k}^p \bigabs{I \cup II} &\leq \lambda_{m,k}^p \bigabs{
      \set{g_m > \lambda_{m,k} }} \leq c\, 2^{-k} K^p.
  \end{align}
  On the other hand with the weak-$L^\sigma$ estimate for
  $\Masmk$ we gain
  \begin{align*}
    \limsup_{m \to \infty} \Big(\lambda_{m,k}^p \abs{III}\Big) &=
    \limsup_{m \to \infty} \Big( \lambda_{m,k}^p \bigabs{
      \set{\alpha_{m,k} \Masmk(\chi_{\frac13 \Qz}\abs{ {\bfh_{2,m}}}) >
        \tfrac 12 \lambda_{m,k}} } \Big)
    \\
    &\leq \limsup_{m \to \infty} \Big( c\,\lambda_{m,k}^{p}
    \bignorm{ {\bfh_{2,m}}}_{L^\sigma(\frac13 \Qz)}^\sigma
    (\alpha_{m,k}/\lambda_{m,k})^\sigma \Big).
  \end{align*}
  Since $2^{2^k+1} \leq \lambda_{m,k} \leq 2^{2^{k+1}}$, 
  $\alpha_{m,k} = \lambda_{m,k}^{2-p}$
  and  {$\bfh_{2,m} \to 0$ in
  $L^\sigma(\frac12 \Qz)$ (which is a consequence of $\bfH_{2,m} \to 0$ in
  $L^\sigma(\frac12 \Qz)$ and Corollary \ref{cor:divdiv})}, it follows that
  \begin{align*}
    \limsup_{m \to \infty} \Big(\lambda_{m,k}^p \abs{III}\Big) &= 0.
  \end{align*}
  This and~\eqref{eq:4.17b} prove~\ref{itm:wlw:levelset}.
\end{proof}

\begin{theorem}
  \label{thm:ulwhole}
  Let $1<p< \infty$ with $p,p'>\sigma$.  Let $\bfu_m$ and $\bfG_m$
  satisfy $\partial_t \bfu_m = - \divergence \bfG_m$ in the sense of
  distributions~$\mathcal{D}_{\divergence}'( \Qz)$.  Assume
  that $\bfu_m$ is a weak null sequence in $L^p(\Iz; W^{1,p}(\Bz))$
  and a strong null sequence in~$L^\sigma(\Qz)$  {and bounded in $L^\infty(\Iz,T;L^\sigma(\Bz))$}. Further assume that
  $\bfG_m = \bfG_{1,m} + \bfG_{2,m}$ such that $\bfG_{1,m}$ is a weak
  null sequence in $L^{p'}(\Qz)$ and $\bfG_{2,m}$ converges strongly
  to zero in $L^\sigma(\Qz)$. Then there is a double sequence
  $(\lambda_{m,k})\subset \R^+$ and $k_0 \in \setN$ with
  \begin{enumerate}[label=(\alph{*}),start=1]
  \item \label{itm:wlu:lmk} $2^{2^k} \leq \lambda_{m,k}\leq
    2^{2^{k+1}}$
  \end{enumerate}
  such that the double sequences $\bfu_{m,k} :=
  \bfu^{\alpha_{m,k}}_{\lambda_{m,k}} \in L^1(\Qz)$, $\alpha_{m,k} :=
  \lambda_{m,k}^{2-p}$ and $\mathcal{O}_{m,k} :=
  \mathcal{O}^{\alpha_{m,k}}_{\lambda_{m,k}}$ (defined in
  Theorem~\ref{thm:wlwhole}) satisfy the following properties for all
  $k \geq k_0$
  \begin{enumerate}[label=(\alph{*}),start=2]
  \item \label{itm:wlu:ns} $\umk\in L^s(\frac
    14\Iz;W^{1,s}_{0,\divergence}(\frac 16\Bz))$ for all $s<\infty$
    and $\support(\bfu_{m,k}) \subset \frac 16 \Qz$.
  \item \label{itm:wlu:setne} $\umk= \bfu_m$ a.e. on
    $ \frac 18 \Qz\setminus\mathcal{O}_{m,k}$.
   {\item \label{itm:wlu:n2BMO} $\|\nabla\umk\|_{L^\infty(\frac 14 \Qz)}\leq c\,\lambda_{m,k}$,
  \item \label{itm:wlu:nweak} $\umk \rightarrow 0$ in
    $L^\infty(\frac 14 \Qz)$ for $m \to \infty$ and $k$ fixed.
  \item \label{itm:wlu:n2weak} $\nabla \umk \rightharpoonup^* 0$ in
    $L^\infty(\frac 14 \Qz)$ for $m \to \infty$ and $k$ fixed.}
  \item \label{itm:wlu:levelset} $\displaystyle{
      \limsup_{m\rightarrow\infty}\lambda_{m,k}^p
      \abs{\mathcal{O}_{m,k}} \leq c\,2^{-k}.}$
  \item
    \label{itm:wlu:time}$\displaystyle{\limsup_{m\rightarrow\infty}
      \bigg|\int \bfG_m:\nabla\umk \dxt\bigg| \leq c\,\lambda_{m,k}^p
      \abs{\mathcal{O}_{m,k}}}$
  \end{enumerate}
\end{theorem}
\begin{proof} 
  We define pointwise in time on $\Iz$
  \begin{align*}
    \tilde{\bfu}_m&:=\gamma\bfu_m-\Bog_{\Bz \setminus \frac 12
      \Bz}(\nabla\gamma\cdot \bfu_m),
    \\
    \bfw_m&:=\curl^{-1}\tilde{\bfu}_m,
    \\
    \bfz_m &:=\Delta\Delta_{\frac 12 \Qz}^{-2}\Delta\bfw_m,
  \end{align*}
  where  $\gamma\in C^\infty_0(\Qz)$ with $\chi_{\frac 12
    \Qz}\leq\gamma\leq \chi_{\Qz}$. Then we apply Theorem
  \ref{thm:wlwhole} to the sequence $\bfz_m$. Finally, let
  \begin{align}
    \label{eq:def-umk}
    \umk:=\curl(\zeta\zmk)+\curl(\zeta(\bfw_m-\bfz_m)),
  \end{align}
  where $\zeta\in C^\infty_0(\frac 16 \Qz)$ with $\chi_{\frac 18
    \Qz}\leq\zeta\leq \chi_{\frac 16 \Qz}$. This means on $\frac 18
  \Qz$ we have 
  \begin{align*}
    \umk&= \bfu_m + \curl(\zmk-\bfz_m).
  \end{align*}
  Note that $\curl(\bfw_m - \bfz_m)$ is harmonic (in space) on $\frac
  12 \Qz$  {and bounded in time due to the assumption that $\bfu_m$ is bounded uniformly in $L^\infty(\Iz;L^\sigma(\Bz))$ which transfers to $\bfw_m$ and $\bfz_m$ by Lemma \ref{lem:defw} and \ref{lem:defz}}. This allows us to estimate the higher order spaces
  derivatives on $\frac 14 \Qz$ by lower order ones on~$\frac 12 \Qz$.
  This,~\eqref{eq:def-umk} and Theorem \ref{thm:wlwhole} immediately
  imply all claimed properties except
  \ref{itm:wlu:time}.

  The claim of~\ref{itm:wlu:levelset}
  follows exactly as
~\ref{itm:wlw:levelset}
  of Theorem~\ref{thm:wlwhole}.

  Let us prove~\ref{itm:wlu:time}. It follows by simple density
  arguments, that $\umk$ is an admissible test
  function for the equation $\partial_t \bfu_m = - \divergence
  \bfG_m$. We thus get
  \begin{align*}
    \lefteqn{\int\bfG_m:\nabla\umk \dxt} \quad & 
    \\
    &=\int
    \partial _t \bfu_m
    \,\umk \dxt
    \\
    &=\int \big(
    \partial _t \curl\bfw_m\big)\curl\big(\zeta\zmk
    \big) \dxt
    +\int \big(
    \partial _t \curl\bfw_m\big)\curl\big(\zeta\big(\bfw_m-\bfz_m
    \big)\big) \dxt
    \\
    &=-\int \big(
    \partial _t \bfz_m\big)\Delta\big(\zeta
    \zmk\big) \dxt -\int \big(
    \partial _t\bfz_m\big)\Delta\big(\zeta\big(\bfw_m-\bfz_m
    \big)\big) \dxt
    \\
    &=:T_1+T_2.
  \end{align*}
  Here we took into account $\curl \curl \bfw_m = -\Delta \bfw_m$ (due
  to $\divergence \bfw_m=0$) and $\Delta\bfw_m=\Delta\bfz_m$. By
  assumption $\bfG_m$ is bounded in $L^\sigma(\Qz)$. 

  Using regularity
  properties of harmonic functions (for $\bfw_m - \bfz_m$) as well as
  Lemma \ref{lem:zi} and Lemma \ref{lem:defw} we get (after choosing
  a subsequence)
  \begin{align*}
    \bigg(\dashint_{\Qz} |\Delta\big(\zeta\big(\bfw_m-\bfz_m
    \big)\big)\big|^{\sigma'}\dxt\bigg)^{\frac 1{\sigma'}}&\leq c\,
    r_0^{-2} \bigg(\dashint_{\frac 14 \Qz}|\bfw_m-\bfz_m
    \big|^{\frac{3\sigma}{3+\sigma}}\dxt\bigg)^{\frac{3+\sigma}{3\sigma}}
    \\
    &\leq c\,r_0^{-2} \bigg(\dashint_{\frac 12 \Qz}|\bfw_m
    \big|^{\frac{3\sigma}{3+\sigma}}\dxt\bigg)^{\frac{3+\sigma}{3\sigma}}
    \\
    &\leq c\, r_0^{-3} \bigg(\dashint_{\Qz}|\tilde{\bfu}_m
    \big|^{\sigma}\dxt\bigg)^{\frac{1}{\sigma}}\longrightarrow 0
    \quad\text{as} \quad m\rightarrow\infty.
  \end{align*}
  Since additionally $\partial_t \bfz_m$ is uniformly bounded
  in~$L^\sigma(\frac 12 \Qz)$ by Lemma~\ref{lem:zi}, we get $T_2 \to
  0$ as $m \to \infty$.

  Furthermore, we have
  \begin{align*}
    T_1=\int\big(\partial_t \big(
    \bfz_m-\zmk\big)\big)\Delta\big(\zeta
    \zmk\big) \dxt+\int \big(
    \partial _t \zmk\big)\Delta\big(\zeta
    \zmk\big) \dxt =: T_{1,1} + T_{1,2},
  \end{align*}
  where the first term can be bounded using Theorem \ref{thm:wlwhole}
  \ref{itm:wlw:time}. So it remains to show that 
  \begin{align*}
    T_{1,2}:=\int \big(
    \partial _t \zmk\big)\Delta\big(\zeta
    \zmk\big) \dxt \longrightarrow 0 \qquad \text{as $m \to \infty$}.
  \end{align*}
  We have
  \begin{align*}
    T_{1,2}&=-\int \frac 12 
    \partial _t (|\nabla \zmk|^2) \zeta \dxt + \int \big(
    \partial _t \zmk\big)\divergence \big(\nabla\zeta\otimes
    \zmk\big) \dxt\\
    &=\int \frac 12 |\nabla \zmk|^2\partial_t\zeta \dxt + \int \big(
    \partial _t \zmk\big)\divergence \big(\nabla\zeta\otimes
    \zmk\big) \dxt
  \end{align*}
  The first term is estimated by
  Theorem~\ref{thm:wlwhole}\ref{itm:wlw:nweak}. For the second we use
   {Lemma~\ref{lem:zlLpW2p} and} Lemma \ref{lem:zi} ($s=\sigma$)
 to find
  \begin{align*}
    \lefteqn{ \int \abs{\partial _t \zmk}\abs{\divergence
        \big(\nabla\zeta\otimes \zmk\big)} \dxt}
    \\
    &\leq c\bigg(\int_{\frac13Q_0}
    \abs{\bfG_m}^\sigma+\abs{\nabla^2\bfz_m}^\sigma\dxt\bigg)^\frac1\sigma
    \bigg(\int_{\frac13Q_0} {\abs{\nabla\bfz_{m,k}}^{\sigma'}+\abs{\bfz_{m,k}}^{\sigma'}}
    \dxt\bigg)^\frac1{\sigma'}.
  \end{align*}
  Now because $\bfG_m$ and $\nabla^2\bfz_m$ are uniformly bounded in
  $L^\sigma(\frac12Q_0)$ we find by
  Theorem~\ref{thm:wlwhole}~\ref{itm:wlw:nweak}, that
  \begin{align*}
    \lim_{m\rightarrow\infty}T_{1,2}=0,
  \end{align*} 
  which proves the claim of \ref{itm:wlu:time}.
\end{proof}
The following corollary is useful in the application of the solenoidal
Lipschitz truncation.
\begin{corollary}
  \label{cor:main}
  Let all assumptions of Theorem~\ref{thm:ulwhole} be satisfied with
  $\zeta\in C^\infty_0(\frac 16 \Qz)$ with $\chi_{\frac 18
    \Qz}\leq\zeta\leq \chi_{\frac 16 \Qz}$ as in the proof of
  Theorem~\ref{thm:ulwhole}. If additionally $\bfu_m$ is uniformly
  bounded in $L^\infty(\Iz, L^\sigma(\Bz))$, then for every $\bfK \in
  L^{p'}(\frac 16 Q_0)$
  \begin{align*}
    \limsup_{m\rightarrow\infty} \bigg|\int
    \big(\big(\bfG_{1,m}+\bfK):\nabla\bfu_m \big) \zeta
    \chi_{\mathcal{O}_{m,k}^\complement} \dxt\bigg| \leq c\, 2^{-k/p}.
  \end{align*}
\end{corollary}
\begin{proof}
  It follows from~\ref{itm:wlu:n2weak},~\ref{itm:wlu:levelset}
  and~\ref{itm:wlu:time} of Theorem~\ref{thm:ulwhole}.
  \begin{align}
    \label{eq:tmp1}
    \limsup_{m\rightarrow\infty} \bigg|\int (\bfG_m+\bfK):\nabla\umk
    \dxt\bigg| &\leq c\, \lambda_{m,k}^p \abs{\mathcal{O}_{m,k}} \leq
    c\, 2^{-k}
  \end{align}
  Recall that $\umk=\curl(\zeta\zmk)+\curl(\zeta(\bfw_m-\bfz_m))$. We
  have $\bfz_{m,k}, \nabla \bfz_{m,k} \to 0$ in $L^\infty(\frac 14
  \Qz)$ by Theorem~\ref{thm:wlwhole} for $m \to \infty$ and $k$ fixed.
  Since $\bfu_m$ is a strong null sequence in~$L^\sigma(\Qz)$ and is
  bounded in $L^\infty(\Iz, L^\sigma(\Bz))$ we get $\bfu_m \to 0$
  strongly in $L^s(\Iz, L^\sigma(\Bz))$ for any $s \in (1,\infty)$. By
  continuity of the \Bogovskii{} operator we get the same convergence
  for~$\tilde{\bfu}_m$.  Now, Lemma~\ref{lem:defw} implies $\bfw_m =
  \curl^{-1} \tilde{\bfu}_m \to 0$ in $L^s(\Iz, W^{1, \sigma}(\Rdr))$.
  Using $\bfz_m :=\Delta\Delta_{\frac 12 \Qz}^{-2}\Delta\bfw_m$ and
  Corollary~\ref{cor:mueller} we also get $\bfz_m \to 0$ in
  $L^{s}(\Iz, W^{1, \sigma}(\Rdr))$.  Since $\bfz_m - \bfw_m$ is
  harmonic on $\frac14 \Qz$, we have $\bfz_m -\bfw_m \to 0$ in
  $L^{s}(\Iz, W^{2, s}(\frac 16 \Bz))$. These convergences imply that
  \begin{align*}
    \nabla \umk &= \zeta \nabla \curl \zmk + \bfa_{m,k},
  \end{align*}
  with $\bfa_{m,k} \to 0$ in $L^s(\frac 16 \Qz)$ for $m\to \infty$ and
  $k$ fixed. This, the boundedness of $\bfG_m$ in $L^\sigma(\frac 16
  \Qz)$, $\bfK \in L^{p'}(\frac 16 Q_0)$ and~\eqref{eq:tmp1} imply
  (using $s> \sigma')$
  \begin{align*}
    \limsup_{m\rightarrow\infty} \bigg|\int ((\bfG_m+\bfK):\nabla
    \curl \zmk) \zeta \dxt\bigg| \leq c\, 2^{-k}.
  \end{align*}
  Since $\bfG_m = \bfG_{1,m} + \bfG_{2,m}$, $\bfG_{2,m} \to 0$ in
  $L^\sigma(\frac 16 \Qz)$ and $\bfz_{m,k} \weakto 0$ in
  $L^{\sigma'}(\frac 16 Qz)$ for $m\to \infty$ and $k$ fixed, we get
  \begin{align}
    \label{eq:tmp4}
    \limsup_{m\rightarrow\infty} \bigg|\int ((\bfG_{1,m}+\bfK):\nabla \curl
    \zmk) \zeta \dxt\bigg| \leq c\, 2^{-k}.
  \end{align}
  The boundedness of $\bfG_{1,m}$ and $\bfK$ in $L^{p'}(\frac 16 \Qz)$
  and
  Theorem~\ref{thm:wlwhole}
and~\ref{itm:wlw:levelset}
  prove
  \begin{align*}
    \limsup_{m\rightarrow\infty} \bigg|\int ((\bfG_{1,m}+\bfK):\nabla \curl
    \zmk) \zeta \chi_{\mathcal{O}_{m,k}}\dxt\bigg| &\leq c\, 2^{-k/p}.
  \end{align*}
  This,~\eqref{eq:tmp4} and $\bfz_{m,k} = \bfz_m$ on
  $\mathcal{O}_{m,k}^\complement$ yield
  \begin{align*}
    \limsup_{m\rightarrow\infty} \bigg|\int ((\bfG_{1,m} + \bfK):\nabla \curl
    \bfz_m) \zeta \chi_{\mathcal{O}_{m,k}^\complement}\dxt\bigg| &\leq
    c\, 2^{-k/p}.
  \end{align*}
  Recall that $\bfz_m -\bfw_m \to 0$ in $L^{s}(\Iz, W^{2, s}(\frac 16
  \Bz))$ for any~$s \in (1,\infty)$. This and the boundedness of
  $\bfG_{1,m}$ in $L^{p'}(Q_0)$ allows us to exchange $\bfz_m$ in the
  previous integral by $\bfw_m$. Now $\curl \bfw_m = \bfu_m$ proves
  the claim.
\end{proof}



\begin{remark}[The higher dimensional case]
  For general dimensions, the solenoidal Lipschitz truncation is
  best understood in terms of differential forms.  We start with
  $\tilde{\bfu}$ as given in~\eqref{eq:uH2}. Now, we have to find
  $\bfw$ such that $\curl \bfw = \tilde{\bfu}$ and $\divergence
  \bfw=0$. Let us define the 1-form $\alpha$ on $\setR^n$ associated
  to the vector field $\tilde{\bfu}$ by $\alpha :=\sum_i\tilde{u}_i
  dx^i$. Then we need to find a 2-form $\omega$ such that $d^* \omega
  = \alpha$ and $d \omega = 0$, where $d$ is the outer derivative and
  $d^*$ its adjunct by the scalar product for $k$-forms. Similar to
  $\bfw = \curl^{-1} \tilde{\bfu} = \curl \Delta^{-1} \tilde{\bfu}$ we
  get~$\omega$ by $\omega := d \Delta^{-1} \alpha$. Since we are on
  the whole space, $\Delta^{-1}$ can be constructed by mollification
  with $c\, \abs{x}^{2-n}$. Thus, we have
  \begin{align*}
    \omega(x)= (d\Delta^{-1} \alpha)(x) =
    d\sum_i\bigg(\int_{\setR^n}\frac{\bfu_i(y)}{\abs{x-y}^{d-2}}dy\bigg)
    dx^i.
  \end{align*}

  Let us explain how to substitute the equation $\partial_t \Delta
  \bfw = - \curl \divergence \bfG$, see~\eqref{eq:wttmp}. Instead of
  test functions $\bfpsi$ with $\divergence \bfpsi =0$ we use the
  associated 1-forms $\beta = \sum_i \psi_i dx^i$ with $d^*
  \beta=0$. Thus there exists a 2-form $\gamma$ with $d^*
  \gamma=0$. Then
  \begin{align*}
    \skp{\partial_t \tilde{\bfu}}{\bfpsi} = \skp{\partial_t
      \alpha}{\beta} = \skp{\partial_t d^* \omega}{d^* \gamma} =
    \skp{\partial_t d d^* \omega}{\gamma} = \skp{-\partial_t \Delta
      \omega}{\gamma}, 
  \end{align*}
  where we used $-\Delta=dd^*+d^*d$ and $d\alpha=0$ in the last step.
  Note that $-\Delta$ applied to the form $\omega$ is the same as
  $-\Delta$ applied to the vector field of all components of~$\omega$.
  Now we define $\bfw$ as the associated vector field (with
  $\binom{n}{2}$ components) of $\omega$ and we arrive again at an
  equation for~$\partial_t \Delta\bfw$. This concludes the
  construction; the rest can be done exactly as for dimension
  three. The restriction $p>\frac 65$ in Section~\ref{sec:appl} will
  change to $\frac{2n}{n+2}$.

\end{remark}

\section{Application to generalized Newtonian fluids}
\label{sec:appl}

In this section we show how the solenoidal Lipschitz truncation can be
used to simplify the existence proof for weak solutions of the power
law fluids. We are able to work completely in the pressure free
formulation.
\begin{theorem}
  \label{thm:pfluid}
  Let $p > \frac 65$, $Q:=(0,T)\times\Omega$, $\bff\in L^{p'}(Q)$ and
  $\bfv_0\in L^2(\Omega)$. Then there is a
  solution $\bfv\in L^{\infty}(0,T; L^2(\Omega)) \cap
  L^p(0,T;W^{1,p}_{0,\divergence}(\Omega))$ to
  \begin{align}
    \begin{aligned}
      \int_{Q} \bfS(\ep(\bfv)):\ep(\bfvarphi)\dxt&=\int_{Q}
      \bff\cdot\bfvarphi\dxt+\int_{Q} \bfv\otimes
      \bfv:\ep(\bfvarphi)\dxt
      \\
      &\quad +\int_{Q} \bfv \,\partial_t\bfvarphi\dxt+\int_{\Omega}
      \bfv_0\,\bfvarphi(0)\dx
    \end{aligned}\label{eq99}
  \end{align}
  for all $\bfphi\in
  C^\infty_{0,\divergence}([0,T)\times\Omega)$.
\end{theorem}
\begin{proof}
  We start with an approximated system whose solution is known to
  exist.  Let $\bfv_m\in L^q(I;W^{1,q}_{0,\divergence}(B))\cap
  L^{\infty}(I; L^2(B))$ be a solution to
  \begin{align}
    \label{appr1}
    \lefteqn{\int_{Q} \bfS(\ep(\bfv)):\ep(\bfvarphi)\dxt+
      \frac{1}{m}\int_{Q}|\ep(\bfv)|^{q-2} \ep(\bfv):\ep(\bfphi)\dxt}
    &\quad
    \\
    &= \int_{Q} \bff\cdot\bfvarphi\dxt +\int_{Q} \bfv\otimes
    \bfv:\ep(\bfvarphi)\dxt+\int_{Q} \bfv
    \,\partial_t\bfvarphi\dxt+\int_{\Omega} \bfv_0\,\bfvarphi(0)\dx
    \nonumber
  \end{align}
  for all $\bfphi\in
  C^\infty_{0,\divergence}([0,T)\times\Omega)$, where
  $q>\max\set{\frac{5p}{5p-6},p}$.

  Due to the choice of $q$ the space of test functions coincides with
  the space where the solution is constructed and the convective term
  becomes a compact perturbation. The existence of $\bfv_m$ is
  therefore standard and can be proved by monotone operator theory.
  Since we are allowed to test with $\bfv_m$, we find
  \begin{align}
    \quad \frac {1} {2}\|\bfv_m (t)\|^2_{L^2}\,+\, \int_0^t
    \int_{{\Omega} } \bfS&(\ep(\bfv_m )) : \ep (\bfv_m ) \dx
    \ds\nonumber+ \frac{1}{m}\int_0^t \int_{{\Omega}
    }|\ep(\bfv_m)|^q\dx\ds
    \\
    \,&=\, \frac {1} {2}\|\bfv_0\|^2_{L^2}\,+\, \int_0^t \int_\Omega
    {\bf f}: \bfv_m\, dx\ds. \hspace*{0.3cm}
    \label{4/EQ:E/APP}
  \end{align}
  for all $t\in (0,T)$.  By coercivity and Korn's inequality we get
  $$\int_Q  \bfS(\ep(\bfv_m )) : \ep (\bfv_m )\dxt\geq c\bigg(\int_Q
  |\nabla\bfv_m|^{p}\dxt-1\bigg)$$ thus
  \begin{align}\label{4/ES:A/Li(H):Lq(V)}
    \,\|m^{-1/q}\ep(\bfv_m) \|_{q,Q}+ \|\bfv_m \|^2_{L^\infty (0,T;
      L^2)} \,+\,\|\nabla\bfv_m \|_{p,Q} \,\leq\, c.
  \end{align}
  Hence we find a function $\bfv\in
  L^p(0,T;W^{1,p}_{0,\divergence}(\Omega))\cap L^\infty(0,T;
  L^2(\Omega)) $ such that for a subsequence
  \begin{align}
    \label{4/LIM/u_eps}
    \begin{alignedat}{2}
      \nabla\bfv_m & \rightharpoonup \nabla\bfv &&\quad\text{in }
      L^{p}(Q),
      \\
      \bfv_m & \weakastto \bfv&&\quad\text{in }
      L^\infty(0,T;L^2({\Omega})),
      \\
      \frac{1}{m}|\ep(\bfv_m)|^{q-2}\ep(\bfv_m) & \rightarrow0
      &&\quad\text{in } L^{q'}(Q),
    \end{alignedat}
  \end{align}
  Since $\bfS(\ep(\bfv_m))$ is bounded in $L^{p'}(Q)$ by
  (\ref{4/ES:A/Li(H):Lq(V)}), there exist $\tilde{\bfS}\in L^{p'}(Q)$
  with
  \begin{align}
    \label{4/LIM/S}
    \bfS(\ep(\bfv_m))& \rightharpoonup \tilde{\bfS} \quad\text{in }
    L^{p'}(Q).
  \end{align}
  Let us have a look at the time derivative. From equation
  (\ref{appr1}) we get the uniform boundedness of $\partial_t \bfv_m$
  in $L^{p}(0,T;(W^{3,2}_{0,\divergence}(\Omega))^*)$ and weak
  convergence of $\partial_t \bfv_m$ to $\partial_t \bfv$ in the same
  space (for a subsequence). This shows by using the compactness of
  the embedding $W^{1,p}_{0,\divergence}({\Omega})\hookrightarrow
  L^{2\sigma_2}_{\divergence}({\Omega})$ for some $\sigma_2>1$ (which
  follows from our assumption~$p> \frac 65$, resp. $p>
  \frac{2n}{n+2}$) and the Aubin-Lions theorem\cite{Lio69} that
  $\bfv_m\rightarrow \bfv$ in
  $L^\sigma(0,T,L^{2\sigma_2}_{\divergence}(\Omega))$. This and the
  boundedness in $L^\infty((0,T); L^2(\Omega))$ imply that for some
  $\sigma>1$
  \begin{align}
    \label{umstrong}
    \bfv_m & \rightarrow \bfv\quad\text{in }
    L^s(0,T;L^{2\sigma}(\Omega))\quad\text{for all }s<\infty.
  \end{align}
  As a consequence we have
  \begin{align}
    \label{umstrong2}
    \bfv_m\otimes \bfv_m & \rightarrow \bfv\otimes\bfv\quad\text{in }
    L^s(0,T;L^{\sigma}(\Omega))\quad\text{for all }s<\infty.
  \end{align}
  Overall, we get our limit equation
  \begin{align}
    \begin{aligned}
      \int_{Q} \tilde{\bfS}:\ep(\bfvarphi)\dxt&=\int_{Q}
      \bff\cdot\bfvarphi\dxt+\int_{Q} \bfv\otimes
      \bfv:\ep(\bfvarphi)\dxt
      \\
      &\quad +\int_{Q} \bfv \,\partial_t\bfvarphi\dxt+\int_{\Omega}
      \bfv_0\,\bfvarphi(0)\dx
    \end{aligned} \label{eq:limit-eq}
  \end{align}
  for all $\bfphi\in
  C^\infty_{0,\divergence}([0,T)\times\Omega)$.

  All of the forthcoming effort is to prove
  $\tilde{\bfS}=\bfS(\ep(\bfv))$ almost everywhere. We start with the
  difference of the equation of $\bfv_m$ and the limit equation.
  \begin{multline}
    -\int_{Q} (\bfv_m -\bfv) \cdot \partial _t \bfvarphi \dx \dt +
    \int_{Q} ({\bf S}(\ep(\bfv_m )) - \tilde{{\bf S}}): \nabla
    \bfvarphi\dx \dt
    \label{4/ID:I/APP:P}
    \\
    \,=\, \int_{Q} \big (\bfv_m \otimes \bfv_m - \bfv\otimes \bfv
    +m^{-1}|\ep(\bfv_m)|^{q-2}\ep(\bfv_m)\big ): \nabla \bfvarphi\dx \dt
  \end{multline}
  for all $\bfvarphi \in C^\infty_{0,\divergence}([0,T)\times \Omega)$.
  We define $ \bfu_m:= \bfv_m - \bfv$. Then by \eqref{umstrong}
  \begin{align}
    \label{4/LIM:s/v_m}
    \begin{aligned}
      &\bfu_m \rightharpoonup \bfzero \quad \quad \text{in }
      L^p(0,T;W^{1,p}_{0,\divergence}(\Omega)),
      \\
      &\bfu_m \rightarrow \bfzero \quad \quad \text{in } L^{2\sigma}(Q).
      \\
      &\bfu_m \weakastto \bfzero \quad \quad \text{in } L^\infty(0,T;
      L^2(\Omega)).
    \end{aligned}
  \end{align}
  Thus, we can write (\ref{4/ID:I/APP:P}) as
  \begin{align}
    \label{eq:limit}
    \int\limits_{(0,T)\times\Omega} \bfu_{m}
    \cdot\partial_t\bfvarphi\dx\, \dt \,=\,
    \int\limits_{(0,T)\times\Omega} \bfH_{ m} : \nabla \bfvarphi \dxt
  \end{align}
  for all $\bfvarphi \in C^\infty_{0,\divergence}(Q)$, where
  $\bfH_m:={\bfH}_{1, m}+\bfH_{2, m}$ with
  \begin{align*}
    {\bfH}_{1, m}&:= {\bf S}( \ep(\bfv_m))-{\bf \tilde{S}},
    \\
    \bfH_{2, m}&:=\bfv_m\otimes \bfv_m - \bfv\otimes \bfv
    +m^{-1}|\ep(\bfv_m)|^{q-2}\ep(\bfv_m).
  \end{align*}
  Moreover,
  \eqref{4/LIM/u_eps} and (\ref{umstrong}) imply
  \begin{align}
    \label{H1tilde}
    \|{\bfH}_{1, m}\|_{p'}\leq c
  \end{align}
  as well as
  \begin{align}
    \label{H2tilde}
    \bfH_{2, m}\rightarrow0\quad\text{in } L^{\sigma}(Q).
  \end{align}
  Now take any cylinder $Q_0\compactsubset(0,T)\times\Omega$.
  Now,~\eqref{4/LIM:s/v_m}, \eqref{eq:limit}, \eqref{H1tilde}
  and~\eqref{H2tilde} ensure that we can apply
  Corollary~\ref{cor:main}. In particular, for suitable $\zeta\in
  C^\infty_0(\frac 16 \Qz)$ with $\chi_{\frac 18 \Qz}\leq\zeta\leq
  \chi_{\frac 16 \Qz}$ Corollary~\ref{cor:main} implies
  \begin{align*}
    \limsup_{m\rightarrow\infty} \bigg|\int \big((\bfH_{1,m}+ {\bf
      \tilde{S}} - \bfS(\ep(\bfv))):\nabla (\bfv_m - \bfv) \big) \zeta
    \chi_{\mathcal{O}_{m,k}^\complement} \dxt\bigg| \leq c\, 2^{-k/p}.
  \end{align*}
  In other words
  \begin{align*}
    \limsup_{m\rightarrow\infty} \bigg|\int
    \Big(\big(\bfS(\ep(\bfv_m))-\bfS(\ep(\bfv))\big):\nabla(\bfv_m - \bfv)
    \Big) \zeta \chi_{\mathcal{O}_{m,k}^\complement} \dxt\bigg| \leq
    c\, 2^{-k/p}.
  \end{align*}
  Let $\theta \in (0,1)$. Then by H{\"o}lder's inequality and
  Theorem~\ref{thm:ulwhole}
  and~\ref{itm:wlu:levelset}
  \begin{align*}
    \lefteqn{\limsup_{m\rightarrow\infty} \int
      \Big(\big(\bfS(\ep(\bfv_m))-\bfS(\ep(\bfv))\big):\nabla(\bfv_m -
      \bfv)\Big)^\theta \zeta \chi_{\mathcal{O}_{m,k}}\dxt} \qquad&
    \\
    &\leq c\, \limsup_{m \to \infty} |\mathcal{O}_{m,k}|^{1-\theta} \leq
    c\,2^{-(1-\theta)\frac kp}. \hspace{30mm}
  \end{align*}
  This, the previous estimate and H{\"o}lder's inequality give
  \begin{align*}
    \limsup_{m\rightarrow\infty} \int
    \Big(\big(\bfS(\ep(\bfv_m))-\bfS(\ep(\bfv))\big):\nabla(\bfv_m -
    \bfv)\Big)^\theta \zeta \dxt &\leq c\,2^{-(1-\theta)\frac kp}.
  \end{align*}
  For $k\to \infty$ the right hand side converges to zero.  Now, the
  monotonicity of $\bfS$ implies that
  $\bfS(\ep(\bfv_m))\rightarrow\bfS(\ep(\bfv))$ a.e. on $\frac 18
  \Qz$. This concludes the proof of Theorem~\ref{thm:pfluid}.

\end{proof}

\section{Solenoidal truncation -- stationary case}
\label{sec:sol_stat}

In this section we show how the solenoidal Lipschitz truncation based
on the $\curl$-representation works in the stationary case. This
provides a simplified approach to the solenoidal Lipschitz truncation
of Ref.~\cite{BreDieFuc12}, which was based on local \Bogovskii{}
projections.

Let us start with a ball $B \subset \Rdr$ and $\bfu \in
W^{1,s}_{0,\divergence}(B)$ with $s \in (1, \infty)$. Since $\bfu$ is
solenoidal, we can define $\bfw := \curl^{-1}\bfu$.  According
to~\eqref{eq:curlmdivfree}, \eqref{eq:curl1} and~\eqref{eq:curl2} we
have $\bfw \in W^{2,s}(\Rdr)$ with $\norm{\nabla \bfw}_s \leq
\norm{\bfu}_s$ and $\norm{\nabla^2 \bfw}_s \leq \norm{\nabla^2
  \bfu}_s$ and $\divergence \bfw = 0$. Since $\curl \curl \bfw = -
\Delta \bfw + \nabla \divergence \bfw = - \Delta \bfw$ and $\curl
\bfw= \bfu=0$ on $\Rdr \setminus B$, it follows that $\bfw$ is
harmonic on $\Rdr \setminus B$.

For $\lambda>0$ define $\mathcal{O}_\lambda := \set{M(\nabla^2
  \bfw)>\lambda}$, where $M$ is the standard non-centered maximal
operator, i.e. $Mf(x) := \sup_{B' \ni x} \dashint_{B'} \abs{f} \dy$
(the supremum is taken over all balls~$B' \subset \Rdr$
containing~$x$). As in Section~\ref{sec:soltrunc} there exists a
Whitney covering $\set{Q_i}$ (of balls) of~$\mathcal{O}_\lambda$ with:
\begin{enumerate}[label={\rm (W\arabic{*})}, leftmargin=*]
\item\label{itm:Swhit1} $\bigcup_{i \in \setN}\frac {1} {2} Q_i \,=\,
  \mathcal{O}_\lambda$,
\item\label{itm:Swhit2} for all $i\in \setN$ we have $8
  Q_i \subset \mathcal{O}_\lambda$ and $16 Q_i \cap (\setR^{3+1}\setminus
  \mathcal{O}_\lambda)\neq \emptyset$, 
\item \label{itm:Swhit3} if $ Q_i \cap Q_{j} \neq \emptyset
  $ then $ \frac 12 r_j\le r_i< 2\, {r_j}$ and $\abs{Q_i \cap Q_j}
  \geq c\, \max \set{\abs{Q_j}, \abs{Q_k}}$.
\item \label{itm:Swhit4}  at every point at most
  $120^{3+2}$ of the sets $4Q_i$ intersect,
\end{enumerate}
where $r_i$ is the radius of $Q_i$.

For each $Q_i$ we define $A_i:= \set{ j \,:\, Q_j\cap
  Q_i\neq\emptyset}$. Note that $\# A_i\le 120 ^{3+2}$ and $r_j \sim
r_i$ for all $j \in A_i$. With respect to the covering $\set{Q_i}$
there exists a partition of unity $\set{\phi_i} \subset
C^\infty_0(\Rdr)$ such that
\begin{enumerate}[label={\rm (P\arabic{*})}, leftmargin=*]
\item \label{itm:PP1} $\chi_{\frac{1}{2}Q_i}\leq \phi_i\leq
  \chi_{\frac{2}{3}Q_i}$,
\item \label{itm:PP2} $\sum_j\phi_j=\sum_{j\in A_i}\phi_j=1$ on $Q_i$,
\item \label{itm:PP3} $\abs{\phi_i} + r_i \abs{\nabla \phi_i} + r_i^2
  \abs{\nabla^2 \phi_i} \leq c$.
\end{enumerate}
Then the solenoidal Lipschitz truncation of $\bfu$ is pointwise
defined as
\begin{align}
  \label{eq:Slip1}
  \bfu_\lambda &:=
  \begin{cases}
    \sum_{j\in \N} \curl(\phi_j \bfw_j) &\qquad \text{in
      $\mathcal{O}_\lambda$},
    \\
    \bfu &\qquad \text{elsewhere},
  \end{cases}
\end{align}
with $\bfw_j := \Pi^1_{Q_j} \bfw$, where $\Pi^1_{Q_j}$ denotes the
first order averaged Taylor polynomial\cite{BreSco94,DieR07} on $Q_j$.
We begin with some estimates for~$\bfw$.


\begin{lemma}
  \label{lem:SPoin}
  For all $j \in \setN$ and all $k \in \setN$ with $Q_j \cap Q_k \neq \emptyset$ we have
  \begin{enumerate}[label={\rm (\alph{*})}]
  \item 
    \label{itm:SPoin1}
    $\dashint_{Q_j}\abs{\frac{\bfw-\bfw_j}{r_j^2} } \dx +
    \dashint_{Q_j}\abs{\frac{\nabla(\bfw-\bfw_j)}{r_j} } \dx \leq
    c\, \dashint_{Q_j} \abs{\nabla^2 \bfw}\dx$.
  \item 
    \label{itm:SPoin1b}
      $\dashint_{Q_j} \abs{\nabla^2 \bfw} \,\dx \leq c\,\lambda$.
  \item 
    \label{itm:SPoin2} 
    $\norm{\bfw_j - \bfw_k}_{L^\infty(Q_j)} \leq c
    \dashint_{Q_j}\abs{\bfw-\bfw_j}\dx + c
    \dashint_{Q_k}\abs{\bfw-\bfw_k}\dx$.
  \item \label{itm:SPoin3} $\norm{\bfw_j - \bfw_k}_{L^\infty(Q_j)}
    \leq c\, r_j^2\, \lambda$.
  \end{enumerate}
\end{lemma}
\begin{proof} 
  The first part~\ref{itm:SPoin1} is just a consequence of the
  classical~\Poincare{} estimate and the properties of~$\Pi^1_{Q_j}$,
  see Lemma~3.1 of Ref.~\cite{DieR07}. The second part~\ref{itm:SPoin1b}
  follows from $Q_j \subset 16 Q_j$ and $16 Q_j \cap
  \mathcal{O}_\lambda^\complement \not=\emptyset$, so $\dashint_{16
    Q_j} \abs{\nabla^2 \bfw} \dx \leq \lambda$.

  Part~\ref{itm:SPoin2} follows from the
  geometric property of the~$Q_j$. If $Q_j \cap Q_k
  \neq\emptyset$, then $\abs{Q_j \cap Q_k} \geq c\, \max
  \{\abs{Q_j}, \abs{Q_k}\}$. This and the norm equivalence for linear
  polynomials imply
  \begin{align*}
    \norm{\bfw_j - \bfw_k}_{L^\infty(Q_j)} &\leq c\, \dashint_{Q_j
      \cap Q_k} \abs{\bfw_j - \bfw_k} \dx
    \\
    &\leq c \dashint_{Q_j}\left|\bfw_j-\bfw\right|\,\dx + c
    \dashint_{Q_k}\left|\bfw-\bfw_k\right|\,\dx.
  \end{align*}
  Finally,~\ref{itm:SPoin3} is a consequence of~\ref{itm:SPoin2},
  \ref{itm:SPoin1} and \ref{itm:SPoin1b}.
\end{proof}
\begin{lemma}
  \label{lem:Slocal}
  There exists $c_0>0$ such that $\lambda \geq \lambda_0 := c_0
  (\dashint_B \abs{\nabla \bfu}^s \dx)^{\frac 1s}$ implies
  $\mathcal{O}_\lambda \subset 2B$.
\end{lemma}
\begin{proof}
  Let $x \in \Rdr \setminus 2B$. We have to show that~$x \not\in
  \mathcal{O}_\lambda$.  We will show that $\dashint_{B'}
  \abs{\nabla^2 \bfw}\dx \leq c\, (\dashint_B \abs{\nabla \bfu}^s
  \dx)^{\frac 1s}$ for any ball~$B'$ containing~$x$.

  First, assume that $B' \cap B \neq \emptyset$. Then $\abs{B'} \geq c\,
  \abs{B}$ and
  \begin{align*}
    \dashint_{B'} \abs{\nabla^2 \bfw} \dx &\leq c\,\bigg(
    \frac{\abs{B}}{\abs{B'}} \bigg)^{\frac 1s} \bigg(\dashint_{B}
    \abs{\nabla \bfu}^s \dx\bigg)^{\frac 1s} \leq c\,
    \bigg(\dashint_{B} \abs{\nabla \bfu}^s \dx\bigg)^{\frac 1s},
  \end{align*}
  where we used H{\"o}lder's inequality and $\norm{\nabla^2 \bfw}_s \leq
  c\, \norm{\nabla \bfu}_s$ by~\eqref{eq:curl2}.

  Second, assume that $B' \cap B = \emptyset$. Let $B''$ denote the
  largest ball with the same center as~$B'$ such that $B'' \cap B =
  \emptyset$. Then $\abs{B''} \geq c\, \abs{B}$. Since $\bfw$ is
  harmonic on~$\Rdr \setminus B$, it follows by the interior estimates
  for harmonic functions, $\norm{\nabla^2 \bfw}_s \leq
  c\, \norm{\nabla \bfu}_s$ and $\abs{B''} \geq c\, \abs{B}$ that
  \begin{align*}
    \dashint_{B'} \abs{\nabla^2 \bfw}\dx &\leq c\,
    \bigg(\dashint_{B''} \abs{\nabla^2 \bfw}^s\dx\bigg)^{\frac 1s}
    \leq c\, \bigg(\dashint_{B} \abs{\nabla \bfu}^s\dx\bigg)^{\frac
      1s}.
  \end{align*}
  This proves the claim.
\end{proof}

We can conclude now, that $\bfu_\lambda$ is a global Sobolev
function. 
\begin{lemma}
  \label{lem:Srepr}
  We have for $\lambda \geq \lambda_0$
  \begin{align*}
    \bfu_\lambda-\bfu =  \sum_{j \in \setN}
    \curl(\phi_j (\bfw_j -\bfw)) \in W^{1,1}_0(2B),
  \end{align*}
  where the sum converges in $W^{1,1}_0(2B)$. In particular,
  $\bfu_\lambda \in W^{1,1}_0(2B)$.
\end{lemma}
\begin{proof}
  The proof is similar to the one of
  Refs.~\cite{BreDieFuc12,DieKreSul13}. Note that the convergence
  will be unconditionally, i.e. irrespectively of the order of
  summation. Obviously, the convergence holds pointwise.  Since
  $\lambda \geq \lambda_0$, it follows by Lemma~\ref{lem:Slocal} that
  $Q_j \subset \mathcal{O}_\lambda \subset 2B$. In particular, each
  summand is in $W^{1,1}_0(2B)$.  It remains to prove convergence of
  the sum in $W^{1,1}_0(2B)$ in the gradient norm. We will show
  absolute convergence of the gradients in~$L^1$. The estimates
  for~$\phi_i$ and Lemma~\ref{lem:SPoin}~\ref{itm:SPoin1} imply
  \begin{align*}
    \sum_j \int \abs{\nabla \curl(\phi_j (\bfw_j -\bfw))} \dx & \leq
    c\,\sum_j \int_{Q_j} \abs{\nabla^2 \bfw}\dx \leq c\,
    \norm{\nabla^2 \bfw}_{L^1(2B)}.
  \end{align*}
  Now $\nabla^2 \bfw \in L^s(\Rdr)$ proves the claim.
\end{proof}
The following theorem describes the basic properties of the Lipschitz
truncation. It is a combination of the techniques
of Refs.~\cite{DieMS08,BreDieFuc12,DieKreSul13}.
\begin{theorem}
  \label{thm:Sest}
  If $\bfu \in W^{1,s}_{0,\divergence}(B)$ and $\lambda \geq
  \lambda_0$, then $\bfu_\lambda \in W^{1,\infty}_{0,\divergence}(2B)$
  and
  \begin{enumerate}[leftmargin=1cm,itemsep=1ex,label={\rm (\alph{*})}]
  \item \label{itm:Sest1} $\bfu_\lambda= \bfu$ on $\R^3 \setminus
    \mathcal{O}_\lambda = \R^3 \setminus \{ M(\nabla^2
    \bfw)>\lambda\}$.
  \item \label{itm:Sest2} $\norm{\bfu_\lambda}_q \leq
    c\,\norm{\bfu}_q$ for $1< q < \infty$ provided $\bfu
    \in L^q(B)$.
  \item \label{itm:Sest3} $\norm{\nabla \bfu_\lambda}_q
    \leq c\, \norm{\nabla \bfu}_q$ for $1< q <\infty$
    provided $\bfu \in W^{1,q}_0(B)$.
  \item \label{itm:Sest4} $|\nabla \bfu_\lambda| \leq c\, \lambda
    \chi_{\mathcal{O}_\lambda} + \abs{\nabla \bfu} \chi_{\Rdr
      \setminus \mathcal{O}_\lambda} \leq c\, \lambda$ almost
    everywhere for all $\lambda>0$.
  \end{enumerate}
\end{theorem}
\begin{proof}
  We will use the representation of Lemma~\ref{lem:Srepr}. The
  claim~\ref{itm:Sest1} follows from $\support(\phi_j) \subset
  \mathcal{O}_\lambda$. We estimate as in the proof of
  Lemma~\ref{lem:Slocal} 
  \begin{align*}
    \biggnorm{ \smash{\sum_{j \in \setN}} \curl(\phi_j(\bfw_j -
      \bfw))}_q^q &\leq c\, \sum_j \int_{Q_j} \frac{\abs{\bfw_j -
        \bfw}^q}{r_j^q} + \abs{\nabla (\bfw_j - \bfw)}^q \dx
    \\
    &\leq c\, \sum_j \int_{Q_j} \abs{\nabla \bfw}^q \dx \leq c\,
    \norm{\bfu}_q^q
  \end{align*}
  using the properties of the averaged Taylor polynomial,
  Lemma~3.1 of Ref.~\cite{DieR07}, and~\eqref{eq:curl1}. Similarly,
  with~\eqref{eq:curl2}
  \begin{align*}
    \biggnorm{ \nabla \smash{\sum_{j \in \setN}} \curl(\phi_j(\bfw_j -
      \bfw))}_q^q &\leq c\, \sum_j \int_{Q_j} \frac{\abs{\bfw_j -
        \bfw}^q}{r_j^{2q}} + \frac{\abs{\nabla(\bfw_j -
        \bfw)}^q}{r_j^{q}} + \abs{\nabla^2 \bfw}^q \dx
    \\
    &\leq c\, \sum_j \int_{Q_j} \abs{\nabla^2 \bfw}^q \dx \leq c\,
    \norm{\nabla \bfu}_q^q.
  \end{align*}
  This and the representation of Lemma~\ref{lem:Srepr}
  prove~\ref{itm:Sest2} and~\ref{itm:Sest3}.

  Let us show~\ref{itm:Sest4}. It suffices to verify $\abs{\nabla
    \bfu_\lambda} \leq c\,\lambda$ on $\mathcal{O}_\lambda$, since on
  $\Rdr \setminus \mathcal{O}_\lambda$ we have $\abs{\nabla \bfu} \leq
  M(\nabla \bfu) \leq M(\nabla^2 \bfw) \leq \lambda$. For $k \in
  \setN$ we estimate
  \begin{align*}
    \norm{\nabla \bfu_\lambda}_{L^\infty(Q_k)} &= \biggnorm{ \sum_{j
        \in A_k} \nabla \curl(\phi_j (\bfw_j -
      \bfw_k))}_{L^\infty(Q_k)}
    \\
    &\leq c\, \sum_{j \in A_k} \bigg( r_k^{-2} \norm{\bfw_j -
      \bfw_k}_{L^\infty(Q_j)} + r_k^{-1} \norm{\nabla(\bfw_j -
      \bfw_k)}_{L^\infty(Q_j)} \bigg)
    \\
    &\leq c\, \sum_{j \in A_k} r_k^{-2} \norm{\bfw_j -
      \bfw_k}_{L^\infty(Q_j)} \leq c\,\lambda,
  \end{align*}
  where we used $\sum_{j \in A_k} \phi_j=1$ on~$Q_k$, inverse
  estimates for linear polynomials and
  Lemma~\ref{lem:SPoin}~\ref{itm:SPoin3}. Now $\mathcal{O}_\lambda =
  \bigcup_k Q_k$ proves~\ref{itm:Sest4}.
\end{proof}
The following theorem is an application of the Lipschitz truncation
to weak null sequences. It is similar to the results
in Refs.~\cite{DieMS08,BreDieFuc12,DieKreSul13}, which were used to prove
the existence of weak solutions.
\begin{theorem}
  \label{thm:remlip}
  Let $1< s< \infty$ and let $\bfu_m \in W^{1,s}_{0,\divergence}(B)$
  be a weak $W^{1,s}_{0,\divergence}$ null sequence. Then there exist
  $j_0 \in \mathbb{N}$ and a double sequence $\lambda_{m,j} \in\R$ with
  $2^{2^j} \leq \lambda_{m,j}\leq 2^{2^{j+1}-1}$ such that the
  Lipschitz truncations $\bfu_{m,j} := \bfu_{\lambda_{m,j}}$ have the
  following properties for $j \geq j_0$.
  \begin{enumerate}[leftmargin=1cm,itemsep=1ex,label={\rm (\alph{*})}]
  \item \label{itm:remlip1} $\bfu_{m,j}\in
    W^{1,\infty}_{0,\divergence}(2B)$ and $\bfu_{m,j}=\bfu_m$ on $\Rdr
    \setminus \mathcal{O}_{m,j}$ for all $m \in \setN$,
    \\
    where
    $\mathcal{O}_{m,j} := \set{M(\nabla^2 (\curl^{-1} \bfu_m))>
      \lambda_{m,j}}$.
  \item \label{itm:remlip2} $\|\nabla\bfu_{m,j}\|_\infty\leq
    c\lambda_{m,j}$ for all $m \in \setN$,
  \item  \label{itm:remlip3} $\bfu_{m,j} \to 0$ for $m \to
    \infty$ in $L^\infty(\Omega)$,
  \item \label{itm:remlip4} $\nabla\bfu_{m,j} \weakastto 0$ for $m
    \to \infty$ in $L^\infty(\Omega)$,
  \item \label{itm:remlip5} For all $m,j \in \N$ holds
    $\norm{\lambda_{m,j} \chi_{\mathcal{O}_{m,j}}}_s
    \leq c(q)\, 2^{-\frac{j}{s}}\norm{\nabla \bfu_m}_s$.
  \end{enumerate}
\end{theorem}
\begin{proof}
  The proof is almost exactly as in Refs.~\cite{BreDieFuc12,DieKreSul13}
  once we have Theorem~\ref{thm:Sest}. We only use additionally the
  continuity properties of~$\curl^{-1}$, see~\eqref{eq:curl1} and
  \eqref{eq:curl2}. We therefore omit a detailed proof.
\end{proof}
 {
 \subsection*{Acknowledgement}
 The work of the authors was supported by Leopoldina (German National Academy of Science). The authors wish to thank Erik B\"aumle who found a serious flaw in an earlier version of this paper.
}

\def\polhk#1{\setbox0=\hbox{#1}{\ooalign{\hidewidth
  \lower1.5ex\hbox{`}\hidewidth\crcr\unhbox0}}}
  \def\ocirc#1{\ifmmode\setbox0=\hbox{$#1$}\dimen0=\ht0 \advance\dimen0
  by1pt\rlap{\hbox to\wd0{\hss\raise\dimen0
  \hbox{\hskip.2em$\scriptscriptstyle\circ$}\hss}}#1\else {\accent"17 #1}\fi}
  \def\ocirc#1{\ifmmode\setbox0=\hbox{$#1$}\dimen0=\ht0 \advance\dimen0
  by1pt\rlap{\hbox to\wd0{\hss\raise\dimen0
  \hbox{\hskip.2em$\scriptscriptstyle\circ$}\hss}}#1\else {\accent"17 #1}\fi}
  \def\ocirc#1{\ifmmode\setbox0=\hbox{$#1$}\dimen0=\ht0 \advance\dimen0
  by1pt\rlap{\hbox to\wd0{\hss\raise\dimen0
  \hbox{\hskip.2em$\scriptscriptstyle\circ$}\hss}}#1\else {\accent"17 #1}\fi}
  \def\ocirc#1{\ifmmode\setbox0=\hbox{$#1$}\dimen0=\ht0 \advance\dimen0
  by1pt\rlap{\hbox to\wd0{\hss\raise\dimen0
  \hbox{\hskip.2em$\scriptscriptstyle\circ$}\hss}}#1\else {\accent"17 #1}\fi}
  \def\cprime{$'$}
\providecommand{\bysame}{\leavevmode\hbox to3em{\hrulefill}\thinspace}
\providecommand{\MR}{\relax\ifhmode\unskip\space\fi MR }
\providecommand{\MRhref}[2]{%
  \href{http://www.ams.org/mathscinet-getitem?mr=#1}{#2}
}
\providecommand{\href}[2]{#2}

\end{document}